\documentclass[a4paper,11pt]{article}

\usepackage{amsmath,amsfonts,amssymb,amsthm}
\usepackage{mathtools}
\usepackage{stmaryrd}
\usepackage{latexsym}
\usepackage{cite}
\usepackage[dvips]{graphicx}
\usepackage[small,bf]{caption}
\usepackage{multirow}
\usepackage{float}
\usepackage{subfig}

\usepackage{listings}
\usepackage[table,x11names]{xcolor} 
\lstdefinelanguage{pseudo}{
frame=BT,
mathescape,
morekeywords={1.,2.,3.,4.,5.,6.,7.,8.,9.,10.,11.,12.,(a),(b),(c),(d)},
basicstyle=\normalsize \ttfamily \color{black},
showspaces=false, 
showstringspaces=false,        
showtabs=false,
keywordstyle=\bfseries,
}

\def \Id {\operatorname{Id}}
\def \tr {\operatorname{tr}}

\def \det {\operatorname{det}}
\def \Cof {\operatorname{Cof}}
\def \Hdiv {H(\operatorname{div})}
\def \Hcurl {H(\operatorname{curl})}
\def \etal {\emph{et al.}}

\newcommand{\lTwo}[1]{L^2(\Omega; {{#1}} )}
\newcommand{\hOne}[1]{H^1(\Omega; {{#1}} )}
\newcommand{\hDiv}[1]{H(\operatorname{div},\Omega; {{#1}} )}
\newcommand{\lTwoT}[1]{L^2(\Omega_\theta; {{#1}} )}
\newcommand{\hOneT}[1]{H^1(\Omega_\theta; {{#1}} )}
\newcommand{\hDivT}[1]{H(\operatorname{div},\Omega_\theta; {{#1}} )}
\numberwithin{equation}{section}
\theoremstyle{plain}
\newtheorem{theorem}{Theorem}[section]
\theoremstyle{remark}
\newtheorem{rmrk}[theorem]{Remark}
\theoremstyle{definition}

\theoremstyle{lemma}
\newtheorem{lemma}[theorem]{Lemma}
\usepackage{tikz}
\usetikzlibrary{backgrounds, arrows, calc, fit, patterns, fadings, plotmarks,
  3d, shapes, shapes.geometric, shapes.symbols, shapes.arrows,
  shapes.multipart, shapes.callouts, shapes.misc, circuits, circuits.logic,
  datavisualization, datavisualization.polar,
  datavisualization.formats.functions, er, automata, chains, topaths, trees,
  petri, mindmap, matrix, calendar, folding, through, turtle, positioning,
  scopes, decorations, decorations.fractals, decorations.shapes,
  decorations.text, decorations.pathmorphing, decorations.pathreplacing,
  decorations.footprints, decorations.markings, shadows, intersections,
  fixedpointarithmetic, fpu, svg.path, external, graphs, plothandlers,
  pgfplots.groupplots, snakes}
\usepackage[hang,flushmargin]{footmisc}
\providecommand{\keywords}[1]{{\small{\textbf{Keywords:}} #1}}

\title{
Volumetric expressions of the shape gradient of the compliance in structural shape optimization
}
\author{
M. Giacomini \footnotemark[1]\textsuperscript{ \ ,}\footnotemark[2] , O. Pantz  
\footnotemark[3] \ and K. Trabelsi \footnotemark[2]
}
\date{}

\begin{document}

\maketitle

\renewcommand{\thefootnote}{\fnsymbol{footnote}}

\footnotetext[1]{CMAP, Inria, Ecole polytechnique, CNRS, Universit\'e Paris-Saclay 91128 Palaiseau, France.}
\footnotetext[2]{DRI Institut Polytechnique des Sciences Avanc\'ees, 63 Boulevard de Brandebourg, 94200 Ivry-sur-Seine, France.}
\footnotetext[3]{Universit\'e C\^ote d'Azur, CNRS, LJAD, France.
\vspace{5pt} }

\footnotetext{
\textit{
M. Giacomini is member of the DeFI team at Inria Saclay \^Ile-de-France.
\emph{Current address:} Laboratori de C\`alcul Num\`eric (LaC\`aN), Universitat Polit\`ecnica de Catalunya BarcelonaTech. Jordi Girona 1 i 3, 08034 Barcelona, Spain.}
\vspace{5pt} }

\footnotetext{\textit{e-mail:} \texttt{matteo.giacomini@polytechnique.edu; 
olivier.pantz@unice.fr; karim.trabelsi@ipsa.fr}}

\renewcommand{\thefootnote}{\arabic{footnote}}

\begin{abstract}
In this article, we consider the problem of optimal design of a compliant structure 
under a volume constraint, within the framework of linear elasticity.
We introduce the pure displacement and the dual mixed formulations 
of the linear elasticity problem and we compute the volumetric expressions of the shape 
gradient of the compliance by means of the velocity method.
A preliminary qualitative comparison of the two expressions of the shape gradient is 
performed through some numerical simulations using the Boundary Variation Algorithm.
\end{abstract}
%
%

\section{Introduction}

In his seminal work \cite{Hadamard1907}, Hadamard proposed a strategy to optimize 
a given shape-dependent functional by deforming the domain according to a 
velocity field. 
Within this framework, a key aspect is the choice of an appropriate direction 
that guarantees the improvement of the value of the functional under analysis.
Gradient-based methods for shape optimization are a well-established approach 
for the solution of PDE-constrained optimization problems of shape-dependent 
functionals. In particular, they exploit the information of the so-called 
shape gradient - that is the differential of the objective functional 
with respect to perturbations of the boundary of the shape - to compute 
the aforementioned descent direction.

Several approaches have been proposed in the literature to compute the shape 
gradient. We refer to \cite{giacomini:hal-01201914} and references therein 
for an overview of the existing methods.
The most common strategy relies on an Eulerian approach and provides a 
surface expression of the shape gradient. 
Starting from the boundary representation of the shape gradient, 
it is straightforward to construct an explicit expression for the descent 
direction. As a matter of fact, let the shape gradient of a functional 
$J(\Omega)$ be
$$
\langle dJ(\Omega),\theta \rangle = \int_{\partial\Omega}{h \theta \cdot n \ ds} ,
$$
it follows that $\theta = - h n$ \ on \ $\partial\Omega$ is a descent direction 
for $J(\Omega)$, that is $\theta$ is such that $\langle dJ(\Omega),\theta \rangle < 0$.
\\
An alternative approach for the computation of the shape gradient relies on 
mapping the quantities defined over the perturbed domain to a reference domain and 
differentiating the resulting functional. Following this method, a volumetric 
expression of the shape gradient may be derived.
The resulting expression of the shape gradient is defined on the whole 
domain $\Omega$ and the solution of an additional variational equation is 
required in order to compute the descent direction $\theta$.
\\
Owing to the Hadamard-Zol\'esio structure theorem (cf. \cite{Delfour:2001:SGA:501610}), 
the restriction of $\langle dJ(\Omega),\theta \rangle$ to the space 
$\mathcal{D}(\mathbb{R}^d , \mathbb{R}^d)$ is a vector-valued distribution whose 
support is included in $\partial\Omega$.
Though the two expressions are equivalent in a continuous framework, 
the surface representation of the shape gradient may not exist 
if the boundary of the domain is not sufficiently smooth 
and the corresponding descent direction $\theta$ may suffer from poor regularity. 
The interest of using volumetric formulations of the shape gradient was first 
suggested in \cite{Berggren2010} and later rigorously investigated in 
\cite{HPS_bit} for the case of elliptic problems: in this latter work, the authors 
proved that the volumetric formulation generally provides better numerical accuracy when 
using the Finite Element Method.

In this work, we present a first attempt to derive the volumetric 
expressions of the shape gradient of a shape-dependent functional 
within the framework of linear elasticity.
In particular, we consider a pure displacement formulation and a 
family of dual mixed variational formulations for the linear 
elasticity problem and we analyze the classical problem of minimization 
of the compliance under a volume constraint.
We derive the volumetric expressions of the shape gradient of the 
compliance for both the pure displacement and the dual mixed formulations 
of the linear elasticity problems and we provide a preliminary 
qualitative comparison through some numerical test.

The rest of this article is organized as follows.
In section \ref{ref:linearElas}, we introduce the pure displacement 
formulation of the linear elasticity problem and two dual mixed 
formulations, namely the Hellinger-Reissner one and a variant 
arising from the weak imposition of the symmetry of the stress tensor.
In section \ref{ref:compliane_minimization}, we describe the abstract 
framework of a PDE-constrained optimization problem of a shape-dependent 
functional and we specify it for the case of the minimization of the 
compliance under a volume constraint.
The derivation of the volumetric expressions of the shape gradient 
of the compliance starting from the pure displacement and the 
dual mixed formulations is discussed respectively in sections 
\ref{ref:shapeGrad_H1} and \ref{ref:shapeGrad_mixed}. 
A preliminary comparison of the aforementioned expressions 
by means of numerical simulations is presented in 
section \ref{ref:shapeGrad_comparison}, whereas 
section \ref{ref:conclusion} summarizes our results and 
highlights ongoing and future investigations.

\section{The linear elasticity problem}
\label{ref:linearElas}

In this section, we introduce the governing equations that 
describe the mechanical behavior of a solid 
within the infinitesimal strain theory, that is under the assumption of small 
deformations and small displacements.
For a complete introduction to this subject, we refer the interested reader to 
\cite{gould1993introduction, MR1262126, MR936420}.
\\
Let $\Omega \subset \mathbb{R}^d \ , \ d=2,3$ be an open and connected domain 
representing the body under analysis and 
$\partial\Omega = \Gamma^N \cup \Gamma \cup \Gamma^D$ be such that the 
three parts of the boundary are disjoint and $\Gamma^D$ has positive 
$(d-1)$-dimensional Hausdorff measure.
We describe an elastic structure subject to a volume force $f$, 
a load $g$ on the surface $\Gamma^N$, a free-boundary condition on $\Gamma$ 
and clamped on $\Gamma^D$:
\begin{equation}
\left\{
\begin{aligned}
& - \nabla \cdot \sigma_\Omega = f \ & \ \text{in} \ \Omega \\
& \sigma_\Omega = Ae(u_\Omega) \ & \ \text{in} \ \Omega \\
& \sigma_\Omega n = g \ & \ \text{on} \ \Gamma^N \\
& \sigma_\Omega n = 0 \ & \ \text{on} \ \Gamma \\
& u_\Omega = 0 \ & \ \text{on} \ \Gamma^D
\end{aligned}
\right.
\label{eq:elas_pde}
\end{equation}
In (\ref{eq:elas_pde}), $u_\Omega$ is the displacement field, $\sigma_\Omega$ is the stress 
tensor and $e(u_\Omega) \coloneqq \frac{1}{2} \left( \nabla u_\Omega + \nabla u_\Omega^T \right)$ 
is the linearized strain tensor.
The full set of equations consists of tree conservation laws - i.e. the conservation of 
mass, the balance of momentum and of angular momentum - and a material law that describes the 
relationship among the variables at play and depends on the type of solid under 
analysis. 
In particular, the balance of angular momentum implies the symmetry of the stress tensor, 
that is $\sigma_\Omega$ belongs to the space $\mathbb{S}_d$ of $d \times d$ symmetric matrices.
Moreover, we consider a linear elastic material and we prescribe the so-called 
Hooke's law which establishes a linear dependency between the stress 
tensor and the linearized strain tensor via the fourth-order tensor 
$A : \Omega \rightarrow \mathbb{S}_d$ known as elasticity tensor.
In this work, we restrict to the case of a homogeneous isotropic material, 
whence the elasticity tensor $A$ depends neither on $x$ nor on 
the direction of the main strains.
\\
The mechanical properties of a linear elastic homogeneous isotropic material 
are determined by the pair $(\lambda,\mu)$ - known as first and second Lam\'e constants - 
or alternatively by the Young's modulus $E$ and the Poisson's ratio $\nu$ 
(cf. e.g. \cite{MR1262126}).
Within the range of physically admissible values of these constants, the relationship between stress 
tensor and strain tensor reads as follows
\begin{equation}
\sigma_\Omega = A e(u_\Omega) = 2\mu e(u_\Omega) + \lambda \tr(e(u_\Omega))\Id
\label{eq:elas_hooke_lame} 
\end{equation}
where $\tr(\cdot) \coloneqq \cdot : \Id$ is the trace operator and $:$ is the 
Frobenius product.
We remark that the elasticity tensor exists and is invertible as long as 
$\lambda < \infty$. Within this framework, we may introduce the so-called 
compliance tensor $A^{-1}$ whose application to the stress tensor provides the strain tensor:
\begin{equation}
e(u_\Omega) = A^{-1}\sigma_\Omega = \frac{1}{2\mu} \sigma_\Omega - \frac{\lambda}{2\mu(d\lambda+2\mu)} \tr(\sigma_\Omega)\Id  
\label{eq:elas_hooke_inverse} 
\end{equation}

\begin{rmrk}
It is straightforward to observe that when $\lambda \rightarrow \infty$ 
the divergence of the displacement field in (\ref{eq:elas_hooke_lame}) 
has to vanish, that is, the material under analysis is said to be incompressible.
Within this context, the elasticity tensor does not exist and the compliance tensor is singular.
\label{rmrk:incompressibility}
\end{rmrk}

\subsection{The pure displacement variational formulation}
\label{ref:pureDisp}

A classical formulation of the linear elasticity problem is the so-called pure displacement 
formulation in which we express the stress tensor $\sigma_\Omega$ in terms of $u_\Omega$ using 
(\ref{eq:elas_hooke_lame}) and we seek the displacement field within the Sobolev space $H^1(\Omega;\mathbb{R}^d)$.
Let $f \in H^1(\mathbb{R}^d;\mathbb{R}^d)$ and $g \in H^2(\mathbb{R}^d;\mathbb{R}^d)$. We define the 
following space $V_\Omega$
\begin{equation}
V_\Omega \coloneqq H^1_{0,\Gamma^D}(\Omega;\mathbb{R}^d) = \{ v \in \hOne{\mathbb{R}^d} \ : \ v=0 \ \text{on} \ \Gamma^D \}
\label{eq:elas_H1} 
\end{equation}
and we seek a function $u_\Omega \in V_\Omega$ such that 
\begin{equation}
a_\Omega(u_\Omega,\delta u) = F_\Omega(\delta u) \quad \forall \delta u \in V_\Omega 
\label{eq:elas_classical} 
\end{equation}
where the bilinear form $a_\Omega(\cdot,\cdot) : V_\Omega \times V_\Omega \rightarrow \mathbb{R}$ 
and the linear form $F_\Omega(\cdot) : V_\Omega \rightarrow \mathbb{R}$ read as follows
\begin{equation}
a_\Omega(u_\Omega,\delta u) \coloneqq \int_\Omega{ Ae(u_\Omega) : e(\delta u) \ dx } \quad , \quad  F_\Omega(\delta u) \coloneqq \int_\Omega{ f \cdot \delta u \ dx } + \int_{\Gamma^N}{ g \cdot \delta u \ ds } .
\label{eq:elas_H1_a_F} 
\end{equation}
The coercivity of the bilinear form $a_\Omega(\cdot,\cdot)$ may be proved using 
Korn's inequality (cf. \cite{10.2307/2132902}) and existence and uniqueness 
of the solution of problem (\ref{eq:elas_classical}) follow from the classical 
Lax-Milgram theorem. 

\begin{rmrk}
The elasticity tensor $A$ acts as a coefficient in the pure displacement formulation 
(\ref{eq:elas_classical}) of the linear elasticity problem.
As previously stated, $A$ deteriorates for nearly incompressible materials and 
does not exist in the incompressible limit (cf. remark \ref{rmrk:incompressibility}). 
Hence, stability issues may arise in the nearly incompressible case, whereas in the 
incompressible limit the stress tensor cannot be expressed in terms of the displacement field 
and the pure displacement variational formulation cannot be posed.
Nevertheless, outside these configurations the formulation (\ref{eq:elas_classical}) 
accurately describes the mechanical phenomena under analysis 
and the corresponding approximation via Lagrangian Finite Element functions 
provides optimal convergence rate of the discretized solution to the continuous one 
(cf. e.g. \cite{MR2373954, braess2001finite}). 
\end{rmrk}

\subsection{Mixed variational formulations via the Hellinger-Reissner principle}
\label{ref:mixed}

Besides the aforementioned stability issues, a major drawback of the pure displacement 
variational formulation is the indirect evaluation of the stress tensor which is not 
computed as part of the solution of the linear elasticity problem but may only be derived from
(\ref{eq:elas_hooke_lame}) via a post-processing of the displacement field $u_\Omega$.
A possible workaround for both these issues is represented by mixed variational formulations 
in which the target solution is the pair $(\sigma_\Omega,u_\Omega)$ representing 
respectively the stress and displacement fields.
This family of approaches was first proposed by Reissner in his seminal work \cite{reissner1950on} 
and has known a great success in the scientific community since.
We refer to \cite{MR1077658} for additional information on dual mixed variational formulations of the linear 
elasticity problem whereas a detailed introduction to mixed Finite Element methods may be found in \cite{BoffiBrezziFortin}.

Let us introduce the space 
$\hDiv{\mathbb{S}_d} \coloneqq \{ \tau \in \lTwo{\mathbb{S}_d} \ : \ \nabla \cdot \tau \in \lTwo{\mathbb{R}^d} \}$ 
of the symmetric square-integrable tensors whose row-wise divergence is square-integrable.
Thus, we define the spaces $V_\Omega \coloneqq \lTwo{\mathbb{R}^d}$, 
$\Sigma_\Omega \coloneqq \{ \tau \in \hDiv{\mathbb{S}_d} \ : \ \tau n = g \ \text{on} \ \Gamma^N \ \text{and} \ \tau n = 0 \ \text{on} \ \Gamma \}$ and 
$\Sigma_{\Omega,0} \coloneqq \{ \tau \in \hDiv{\mathbb{S}_d} \ : \ \tau n = 0 \ \text{on} \ \Gamma^N \cup \Gamma \}$ 
and we seek $(\sigma_\Omega,u_\Omega) \in \Sigma_\Omega \times V_\Omega$ such that 
\begin{equation}
\begin{aligned} 
a_\Omega(\sigma_\Omega,\delta\sigma) + & b_\Omega(\delta\sigma,u_\Omega) = 0 \quad & \forall \delta\sigma \in \Sigma_{\Omega,0} \\
& b_\Omega(\sigma_\Omega,\delta u) = F_\Omega(\delta u) \quad & \forall \delta u \in V_\Omega
\end{aligned}
\label{eq:elas_mixed} 
\end{equation}
where the bilinear forms $a_\Omega(\cdot,\cdot):\Sigma_\Omega \times \Sigma_\Omega \rightarrow \mathbb{R}$ and 
$b_\Omega(\cdot,\cdot):\Sigma_\Omega \times V_\Omega \rightarrow \mathbb{R}$ and the linear form 
$F_\Omega(\cdot):V_\Omega \rightarrow \mathbb{R}$ read as
\begin{gather}
a_\Omega(\sigma_\Omega,\delta\sigma) \coloneqq \int_\Omega{ A^{-1}\sigma_\Omega : \delta\sigma \ dx } \quad , \quad  b_\Omega(\sigma_\Omega,\delta u) \coloneqq \int_\Omega{ \left( \nabla \cdot \sigma_\Omega \right) \cdot \delta u \ dx } ,
\label{eq:elas_dual_a_b} \\
F_\Omega(\delta u) \coloneqq - \int_\Omega{ f \cdot \delta u \ dx } .
\label{eq:elas_dual_F}
\end{gather}
Existence and uniqueness of the solution of the dual mixed variational formulation 
(\ref{eq:elas_mixed}) follow from Brezzi's theory on mixed methods \cite{Brezzi1974,BoffiBrezziFortin}.
Moreover, in \cite{Arnold1984} the authors proved that 
stability estimates for the dual mixed variational formulation do not deteriorate, be it 
in the case of nearly incompressible materials or in the incompressible limit 
making this approach feasible for the whole range of values of the Lam\'e constants.

\begin{rmrk}
A major drawback of the previously introduced dual mixed variational formulation lies in 
the difficulty of constructing a pair of Finite Element spaces that fulfill the 
requirements of Brezzi's theory in order to guarantee the stability of the method. 
Several authors have been dealing with this issue in the last forty years.
In \cite{MR1930384}, Arnold and Winther proposed the first stable pair of Finite Element 
spaces for the discretization of the linear elasticity problem in two space dimensions. 
The corresponding three-dimensional case was later discussed in \cite{Adams2005, 10.2307/40234556}.
Owing to the large number of Degrees of Freedom and to the high order of the 
involved polynomials, the construction of the basis functions described in the aforementioned works 
and their implementation in existing Finite Element libraries is extremely complex.
Despite this class of Finite Element functions is the most straightforward way to handle 
the aforementioned problem and some recent works \cite{Carstensen20083014, Carstensen20112903} 
have shown their efficiency from a numerical point of view, 
the Arnold-Winther Finite Element spaces are currently far from being a widely spread standard in 
the community.
To the best of our knowledge, among the most common Finite Element libraries, 
only the FEniCS Project (cf. \cite{ans20553}, \texttt{http://www.fenics.org}) 
provides partial support for the Arnold-Winther functions.
\end{rmrk}

\subsection{A dual mixed variational formulation with weakly enforced symmetry of the stress tensor}
\label{ref:mixed_weak}

As stated in the previous subsection, the stress tensor is sought in a subspace of $\hDiv{\mathbb{S}_d}$.
In \cite{BoffiBrezziFortin}, the authors highlight that the choice of this space 
is strictly connected with the will of strongly imposing conservation laws.
In particular, $\sigma_\Omega$ belonging to the space of square-integrable tensors 
whose row-wise divergence is square-integrable strongly enforces the conservation of 
momentum.
Moreover, the symmetry of the stress tensor is a simplified way of expressing the 
conservation of angular momentum for the system under analysis.
It is well-known that imposing exactly a conservation law is not trivial.
Hence, strongly enforcing a second conservation law by requiring the stress 
tensor to be symmetric is likely to be difficult.
\\
In order to circumvent this issue and before the work \cite{MR1930384} by Arnold and Winther 
appeared, several alternative formulations have been proposed in 
the literature to weakly enforce the symmetry of the stress tensor via a Lagrange multiplier.
Starting from the pioneering work of Brezzi \cite{Brezzi1974} and Fraejis de Veubeke \cite{Veubeke1975}, 
several authors have proposed mixed formulations in which the symmetry of the stress tensor is 
either weakly enforced or dropped (cf. e.g. \cite{MR553347, MR946367, MR1082821}). 
One of the simplest solutions was developed by Arnold, Brezzi and Douglas Jr. in \cite{MR840802} 
via the so-called PEERS element: within this framework, the stress tensor is discretized by means of 
an augmented cartesian product of the Raviart-Thomas Finite Element space, the displacement field 
using piecewise constant functions and the Lagrange multiplier via a $\mathbb{P}^1$ Finite Element function.
Stemming from the idea of the PEERS element, several other approaches have been proposed 
in the literature, e.g. \cite{MR1005064, MR834332, MR954768, MR956898, MR859922, MR1464150}. 
For a complete discussion on this topic, we refer to \cite{MR2449101}.

In this subsection, we rely on a more recent mixed Finite Element method to approximate the problem 
of linear elasticity with weakly imposed symmetry of the stress tensor. In particular, we refer 
to \cite{MR2249345} for the construction of the stable pair of Finite Element spaces in two 
space dimensions, whereas the corresponding three-dimensional case is treated in \cite{MR2336264}.
The choice of this new approach by Arnold and co-workers, instead of the widely used PEERS, 
is mainly due to the simpler discretization arising from the novel method and to the 
possibility of extending it to the three-dimensional case in a straightforward way.
Let $\mathbb{M}_d$ be the space of $d \times d$ matrices and $\mathbb{K}_d$ be the space of 
$d \times d$ skew-symmetric matrices. 
We define the spaces $V_\Omega \coloneqq \lTwo{\mathbb{R}^d}$, 
$Q_\Omega \coloneqq \lTwo{\mathbb{K}_d}$, 
$\Sigma_\Omega \coloneqq \{ \tau \in \hDiv{\mathbb{M}_d} \ : \ \tau n = g \ \text{on} \ \Gamma^N \ \text{and} \ \tau n = 0 \ \text{on} \ \Gamma \}$ and 
$\Sigma_{\Omega,0} \coloneqq \{ \tau \in \hDiv{\mathbb{M}_d} \ : \ \tau n = 0 \ \text{on} \ \Gamma^N \cup \Gamma \}$. 
Moreover, we introduce the space $W_\Omega \coloneqq V_\Omega \times Q_\Omega$.
The extended system obtained from (\ref{eq:elas_mixed}) by relaxing the symmetry condition 
on the stress tensor through the introduction of a Lagrange multiplier reads as follows: we seek 
$(\sigma_\Omega,(u_\Omega,\eta_\Omega)) \in \Sigma_\Omega \times W_\Omega$ such that
\begin{equation}
\begin{aligned} 
a_\Omega(\sigma_\Omega,\delta\sigma) + & b_\Omega(\delta\sigma,(u_\Omega,\eta_\Omega)) = 0 \quad & \forall \delta\sigma \in \Sigma_{\Omega,0} \\
& b_\Omega(\sigma_\Omega,(\delta u,\delta\eta)) = F_\Omega(\delta u) \quad & \forall (\delta u,\delta\eta) \in W_\Omega 
\end{aligned}
\label{eq:elas_mixed_weak} 
\end{equation}
where the bilinear and linear forms have the following expressions:
\begin{gather}
a_\Omega(\sigma_\Omega,\delta\sigma) \coloneqq \int_\Omega{ A^{-1}\sigma_\Omega : \delta\sigma \ dx } 
\quad , \quad  
b_\Omega(\sigma_\Omega,(\delta u,\delta\eta)) \coloneqq \int_\Omega{ \left( \nabla \cdot \sigma_\Omega \right) \cdot \delta u \ dx} + \frac{1}{2\mu} \int_\Omega{ \sigma_\Omega : \delta\eta \ dx } ,
\label{eq:elas_dual_weak_a_b_c} \\
F_\Omega(\delta u) \coloneqq - \int_\Omega{ f \cdot \delta u \ dx } .
\label{eq:elas_dual_weak_F}
\end{gather}
Existence and uniqueness of the solution for this variant of the dual mixed variational 
formulation of the linear elasticity problem with weakly imposed symmetry of the 
stress tensor follow again from Brezzi's theory (cf. \cite{MR840802}). 
\begin{rmrk}
We highlight that if $(\sigma_\Omega,(u_\Omega,\eta_\Omega))$ is solution of 
(\ref{eq:elas_mixed_weak}), then $\sigma_\Omega$ is symmetric and 
$(\sigma_\Omega,u_\Omega) \in \hDiv{\mathbb{S}_d} \times \lTwo{\mathbb{R}^d}$ is solution 
of the original dual mixed formulation of the linear elasticity problem 
with strongly enforced symmetry of the stress tensor discussed in the previous subsection.
Though the infinite-dimensional formulation of the problem featuring weak symmetry is equivalent 
to the one in which the symmetry of the stress tensor is imposed in a strong way, 
the former allows for novel discretization techniques in which the approximation 
$\sigma_\Omega^h$ of the stress tensor $\sigma_\Omega$ is not guaranteed to 
be symmetric, that is $\sigma_\Omega^h$ solely fulfills the following condition
$$
\int_\Omega{ \sigma_\Omega^h : \delta\eta^h \ dx } = 0 \quad \forall \delta \eta^h \in Q_\Omega^h
$$
where $Q_\Omega^h$ is an appropriate discrete space approximating $\lTwo{\mathbb{K}_d}$.
\end{rmrk}

As stated at the beginning of this subsection, several choices are possible for the discrete 
spaces $\Sigma_\Omega^h$, $V_\Omega^h$ and $Q_\Omega^h$ respectively approximating 
$\hDiv{\mathbb{S}_d}$, $\lTwo{\mathbb{R}^d}$ and $\lTwo{\mathbb{K}_d}$.
In the rest of this article, we consider the approach discussed in \cite{MR2249345}, in 
which the stress tensor is approximated by the cartesian product of two pairs of 
Brezzi-Douglas-Marini Finite Element spaces while the displacement field and 
the Lagrange multiplier are both discretized using piecewise constant functions.

\section{Minimization of the compliance under a volume constraint}
\label{ref:compliane_minimization}

In this section, we introduce the problem of optimal design of compliant structures 
within the framework of linear elasticity, that is the construction of the shape that 
minimizes the compliance under a volume constraint.
Let us consider a vector field $\theta \in W^{1,\infty}(\mathbb{R}^d;\mathbb{R}^d)$.
We introduce a transformation $X_\theta : \mathbb{R}^d \rightarrow \mathbb{R}^d$ 
and we define the open subset $\Omega_\theta \subset \mathbb{R}^d$ as $\Omega_\theta = X_\theta(\Omega)$. 
Moreover, we set that $\Gamma^N_\theta = X_\theta(\Gamma^N)$, $\Gamma_\theta = X_\theta(\Gamma)$ and 
$\Gamma^D_\theta = X_\theta(\Gamma^D)$.
The displacement of an initial point $x \in \Omega$ is governed by the following differential equation:
\begin{equation}
\begin{cases}
\displaystyle\frac{d x_\theta}{d t}(t) = \theta(x_\theta(t)) \\
x_\theta(0) = x
\end{cases}
\label{eq:elas_transformation}
\end{equation}
which admits a unique solution $t \mapsto x_\theta(t,x)$ in $C^1(\mathbb{R};\mathbb{R}^d)$.
Owing to (\ref{eq:elas_transformation}), the initial point $x \in \Omega$ is transported 
by the field $\theta$ to the point $x_\theta = X_\theta(x)$ which belongs to the 
deformed domain $\Omega_\theta$.
Moreover, we denote by $D_\theta$ the Jacobian matrix of the transformation $X_\theta$ and by 
$I_\theta = \det D_\theta$ its determinant.
\\
Within the framework of shape optimization, a common choice for the transformation $X_\theta$ 
is a perturbation of the identity map, that is
\begin{equation}
X_\theta = \Id + \theta + o(\theta) \quad , \quad \theta \in W^{1,\infty}(\mathbb{R}^d;\mathbb{R}^d) .
\label{eq:elas_perturbId}
\end{equation}
Hence, $\Omega_\theta = X_\theta(\Omega) = \{ x + \theta(x) \ : \ x \in \Omega \}$ and under the 
assumption of a small perturbation $\theta$, $X_\theta$ is a diffeomorphism and belongs to the 
following space (cf. \cite{allaire2006conception}):
$$
\mathcal{X} \coloneqq \left\{ X_\theta \ : \ (X_\theta - \Id) \in W^{1,\infty}(\mathbb{R}^d;\mathbb{R}^d) \ \text{and} \ (X_\theta^{-1} - \Id) \in W^{1,\infty}(\mathbb{R}^d;\mathbb{R}^d) \right\} .
$$
By exploiting the notation above, we introduce the set of shapes that may be obtained 
as result of a deformation of the reference domain $\Omega$:
\begin{equation}
\mathcal{U}_{\text{def}} \coloneqq \{ \Omega_\theta \ : \ \exists X_\theta \in \mathcal{X} \ , \ \Omega_\theta = X_\theta(\Omega) \} .
\label{eq:elas_Uad_def}
\end{equation}
Let us define the compliance on a deformed domain $\Omega_\theta$ as
\begin{equation}
J(\Omega_\theta) = \int_{\Omega_\theta}{A^{-1}\sigma_{\Omega_\theta} : \sigma_{\Omega_\theta} \ dx_\theta} .
\label{eq:elas_compliance} 
\end{equation}
The shape optimization problem of the compliance under a volume constraint 
may be written as the following PDE-constrained optimization problem of a shape-dependent functional:
\begin{equation}
\min_{\Omega_\theta \in \mathcal{U}_{\text{ad}}} J(\Omega_\theta)
\label{eq:shapeOpt_inf} 
\end{equation}
where the set of admissible domains $\mathcal{U}_{\text{ad}} \subset \mathbb{R}^d$
is the set of shapes in (\ref{eq:elas_Uad_def}) such that
$\sigma_{\Omega_\theta}$ is the stress tensor fulfilling the linear 
elasticity problem (\ref{eq:elas_pde}) on $\Omega_\theta$ and 
the volume $V(\Omega_\theta) \coloneqq |\Omega_\theta|$ is equal to the initial volume $|\Omega|$.

In real-life problems, the optimal design of compliant structures is usually subject to 
additional constraints, either imposed by the end-user (e.g. volume/perimeter \cite{allaire2006conception} 
or stress \cite{Duysinx2008} constraints) or by the manufacturing process (e.g. maximum/minimum thickness 
\cite{Allaire2016} or molding direction \cite{allaire:hal-01242945} constraints).
Several sophisticated strategies (e.g. quadratic penalty and augmented Lagrangian methods) 
may be considered to handle the constraints involved in optimization problems and 
we refer to \cite{nocedal1999numerical} for a thorough introduction to this subject. 
Within the field of shape optimization, an algorithm based on a Lagrangian functional 
featuring an efficient update strategy for the Lagrange multiplier has been proposed in 
\cite{smo-AP}. Several other approaches have known a great success in the literature, e.g. 
the Method of Moving Asymptotes \cite{NME:NME1620240207} and the Method of Feasible 
Directions \cite{VANDERPLAATS1973739}.
In this article, we restrict ourselves to the classical volume constraint and we enforce it through a 
penalty method using a fixed Lagrange multiplier $\gamma$. Thus the resulting unconstrained shape 
optimization problem reads as follows:
\begin{equation}
\min_{\Omega_\theta \in \mathcal{U}_{\text{ad}}} L(\Omega_\theta) \quad , \quad L(\Omega_\theta) \coloneqq J(\Omega_\theta) + \gamma V(\Omega_\theta)
\label{eq:elas_unconstrained} 
\end{equation}
where $J(\Omega_\theta)$ is the compliance (\ref{eq:elas_compliance}), 
$V(\Omega_\theta)$ is the volume of the domain 
and $\mathcal{U}_{ad}$ is the previously defined set of admissible shapes.

\subsection{A gradient-based method for shape optimization}
\label{ref:BVA}

We consider an Optimize-then-Discretize strategy which relies on the 
analytical computation of the gradient of the cost functional which is then discretized to 
run the optimization loop.
In particular, we exploit the so-called Boundary Variation Algorithm (BVA) described in 
\cite{smo-AP}: this method requires the computation of the 
so-called shape gradient which arises from the differentiation of the functional with respect 
to the shape.
A detailed computation of the volumetric expression of the shape gradient for the 
pure displacement and the dual mixed formulations of the linear elasticity problem 
is discussed in sections \ref{ref:shapeGrad_H1} and \ref{ref:shapeGrad_mixed}.
Here, we briefly sketch the aforementioned BVA inspired by Hadamard's boundary variation method.
After solving the linear elasticity equation, we compute the expression of the shape gradient.
Then, a descent direction is identified solving the following variational 
problem: we seek $\theta \in X$, $X$ being an appropriate Hilbert space such that
\begin{equation}
( \theta, \delta\theta )_X + \langle dL(\Omega),\delta\theta \rangle = 0 \qquad 
\forall \delta\theta \in X.
\label{eq:variationalP}
\end{equation}
The resulting information is 
used to deform the domain via a perturbation of the identity map $\Id + \theta$.
\begin{lstlisting}[language=pseudo, escapeinside={/*@}{@*/}, 
caption={The Boundary Variation Algorithm}, 
label=scpt:shape-opt-classic]
Given the domain $\Omega_0$, set $j=0$ and iterate:
1. Compute the solution of the state equation;
2. Compute a descent direction $\theta_j \in X$ solving 
/*@ 
\vspace{-7pt} 
$$
( \theta_j, \delta\theta )_X + \langle dL(\Omega_j),\delta\theta \rangle = 0 \quad \forall \delta\theta \in X \ ;
$$ 
@*/
3. Identify an admissible step $\mu_j$;
4. Update the domain $\Omega_{j+1} = (\Id + \mu_j 
\theta_j)\Omega_j$;
5. Until a stopping criterion is not fulfilled, $j=j+1$ and repeat.
\end{lstlisting}

\vspace{10pt}

\noindent We recall that a direction $\theta$ is said to be a genuine descent 
direction for the functional $L(\Omega)$ if $\langle d L(\Omega),\theta \rangle < 0$ .
It is straightforward to observe that a direction fulfilling this condition 
is such that $L(\Omega)$ decreases along $\theta$, that is $L((\Id + \theta)\Omega) < L(\Omega)$.

\subsubsection{Shape gradient of the volume}
\label{ref:shapeGrad_volume}

In order to apply algorithm \ref{scpt:shape-opt-classic} to solve problem (\ref{eq:elas_unconstrained}), 
the analytical expression of the shape gradient of $L(\Omega_\theta)$ is required. 
We remark that the volume $V(\Omega_\theta)$ is a purely geometrical quantity and does not 
depend on the solution of the state problem. Hence, its shape gradient may be easily computed 
by mapping the integral over the deformed domain $\Omega_\theta$ to the integral over 
the fixed domain $\Omega$ and by differentiating the resulting quantity with respect to $\theta$ 
in $\theta=0$ (cf. e.g \cite{Delfour:2001:SGA:501610}):
\begin{equation}
\langle dV(\Omega), \theta \rangle = \int_\Omega{\nabla \cdot \theta \ dx}.
\label{eq:elas_shapeGrad_volume}
\end{equation}
Moreover, owing to (\ref{eq:elas_shapeGrad_volume}) and to the fact that $\Gamma^N$ and 
$\Gamma^D$ are fixed - that is $\theta \cdot n = 0$ on $\Gamma^N \cup \Gamma^D$ - 
the surface expression of the shape gradient of the volume reads as
\begin{equation}
\langle dV(\Omega), \theta \rangle = \int_\Gamma{\theta \cdot n \ ds} .
\label{eq:surfaceShapeV} 
\end{equation}

In the rest of this article, we will focus on the shape gradient of the compliance. 
In particular, in subsection \ref{ref:surface} we will recall the expression of the 
surface shape gradient of the compliance (\ref{eq:elas_compliance}), 
whereas in sections \ref{ref:shapeGrad_H1} and \ref{ref:shapeGrad_mixed} we will derive 
the volumetric expressions respectively for the pure displacement formulation 
(cf. subsection \ref{ref:pureDisp}) and for the mixed formulations 
(cf. subsections \ref{ref:mixed} and \ref{ref:mixed_weak}) of the linear elasticity problem.

\subsubsection{Surface expression of the shape gradient of the compliance}
\label{ref:surface}

In this subsection we recall the surface expression of the shape gradient of the compliance 
which will be later used in section \ref{ref:shapeGrad_comparison} to perform a preliminary 
numerical comparison with the corresponding volumetric formulations.
In particular, for the compliance we get
\begin{equation}
\langle dJ(\Omega), \theta \rangle = - \int_\Gamma{\left( 2\mu e(u_\Omega) : e(u_\Omega) + \lambda (\tr(e(u_\Omega)))^2 \right) \theta \cdot n \ ds} 
\label{eq:surfaceShapeJ} 
\end{equation}
whereas for the augmented functional $L(\Omega_\theta)$ it follows
\begin{equation}
\langle dL(\Omega), \theta \rangle = \int_\Gamma{\left( \gamma - 2\mu e(u_\Omega) : e(u_\Omega) + \lambda (\tr(e(u_\Omega)))^2 \right) \theta \cdot n \ ds} .
\label{eq:surfaceShapeL} 
\end{equation}
We refer to \cite{allaire2006conception} for a detailed discussion on the derivation of the 
above expression.

\section{Volumetric shape gradient of the compliance via the pure displacement formulation}
\label{ref:shapeGrad_H1}

In the pure displacement formulation (cf. subsection \ref{ref:pureDisp}), 
the stress tensor can be expressed in terms of the displacement field through 
the relationship $\sigma_{\Omega_\theta} = Ae(u_{\Omega_\theta})$. 
Hence, (\ref{eq:elas_compliance}) may be rewritten as 
\begin{equation}
J(\Omega_\theta) = \int_{\Omega_\theta}{ Ae(u_{\Omega_\theta}) : e(u_{\Omega_\theta}) \ dx_\theta } = \int_{\Omega_\theta}{ f \cdot u_{\Omega_\theta} \ dx_\theta } + \int_{\Gamma^N_\theta}{ g \cdot u_{\Omega_\theta} \ ds_\theta } ,
\label{eq:elas_compliance_H1} 
\end{equation}
that is we can equivalently reinterpret the compliance as the work of the external forces applied to the domain $\Omega_\theta$.
Owing to the principle of minimum potential energy for the problem 
(\ref{eq:elas_classical})-(\ref{eq:elas_H1_a_F}) on the domain $\Omega_\theta$ 
and to (\ref{eq:elas_compliance_H1}),
we may write the compliance as follows:
\begin{equation}
J_1(\Omega_\theta) \coloneqq - \min_{u_{\Omega_\theta} \in V_{\Omega_\theta}} \int_{\Omega_\theta}{Ae(u_{\Omega_\theta}) : e(u_{\Omega_\theta}) \ dx_\theta} - 2 \int_{\Omega_\theta}{ f \cdot u_{\Omega_\theta} \ dx_\theta } - 2 \int_{\Gamma^N_\theta}{ g \cdot u_{\Omega_\theta} \ ds_\theta } ,
\label{eq:elas_max_H1}
\end{equation}
where $V_{\Omega_\theta} \coloneqq H^1_{0,\Gamma^D_\theta}(\Omega_\theta;\mathbb{R}^d) = \{ v \in \hOneT{\mathbb{R}^d} \ : \ v=0 \ \text{on} \ \Gamma^D_\theta \}$.
\\
Let $j_1(\theta) \coloneqq J_1(\Omega_\theta)$. We are interested in computing 
the shape gradient of $J_1(\Omega)$, that is
\begin{equation}
\langle dJ_1(\Omega),\theta \rangle \coloneqq \lim_{\theta \searrow 0} \frac{J_1(\Omega_\theta)-J_1(\Omega)}{\theta} = \lim_{\theta \searrow 0} \frac{j_1(\theta)-j_1(0)}{\theta} \eqqcolon j_1'(0) .
\label{eq:elas_minGrad}
\end{equation}
We refer to \cite{Delfour:2001:SGA:501610} 
for a result on the differentiability of a minimum with respect to a parameter. 
Moreover, we remark that the space $V_{\Omega_\theta}$ in (\ref{eq:elas_max_H1}) depends on the parameter 
$\theta$. We use the function space parameterization technique described in 
\cite{Delfour:2001:SGA:501610} to transport the quantities defined on the deformed domain 
$\Omega_\theta$ back to the reference domain $\Omega$. Thus, we are able to rewrite 
(\ref{eq:elas_max_H1}) using solely functions of the space $V_\Omega$ 
which no longer depends on $\theta$ and we apply elementary differential calculus techniques to 
compute the derivative of the objective functional with respect to the parameter $\theta$.

Let us introduce the following transformation to parameterize the functions in 
$H^1_{0,\Gamma^D_\theta}(\Omega_\theta;\mathbb{R}^d)$ in terms of the elements of 
$H^1_{0,\Gamma^D}(\Omega;\mathbb{R}^d)$:
\begin{equation}
\mathcal{P}_\theta : H^1_{0,\Gamma^D}(\Omega;\mathbb{R}^d) \rightarrow H^1_{0,\Gamma^D_\theta}(\Omega_\theta;\mathbb{R}^d) 
\quad , \quad
v_{\Omega_\theta} = \mathcal{P}_\theta(v_\Omega) = v_\Omega \circ X_\theta^{-1} .
\label{eq:elas_transformationH1} 
\end{equation}
\begin{lemma}
Let $u_\Omega \in H^1_{0,\Gamma^D}(\Omega;\mathbb{R}^d)$.
We consider $u_{\Omega_\theta} = \mathcal{P}_\theta(u_\Omega)$ according to 
the transformation (\ref{eq:elas_transformationH1}). It follows that
\begin{equation}
\frac{1}{2} \left( \nabla_{x_\theta} u_{\Omega_\theta} + \nabla_{x_\theta} u_{\Omega_\theta}^T \right) \eqqcolon e_{x_\theta}(u_{\Omega_\theta}) = \frac{1}{2} \left( \nabla_x u_\Omega D_\theta^{-1} + D_\theta^{-T} \nabla_x u_\Omega^T \right)
\label{eq:elas_linStrain}
\end{equation}
where $\nabla_{x_\theta}$ (respectively $\nabla_x$) represents the gradient with respect to 
the coordinate of the deformed (respectively reference) domain.
\label{theo:linStrain}
\end{lemma}
\begin{proof}
Owing to (\ref{eq:elas_transformationH1}), $u_{\Omega_\theta} = u_\Omega \circ X_\theta^{-1}$.
Thus,
$$
\frac{\partial \left( u_{\Omega_\theta} \right)_i}{\partial \left( x_\theta \right)_j} = 
\frac{\partial \left( u_\Omega \right)_i}{\partial \left( x \right)_m} \frac{\partial \left( X_\theta^{-1} \right)_m}{\partial \left( x_\theta \right)_j} =
\frac{\partial \left( u_\Omega \right)_i}{\partial \left( x \right)_m} \left( D_\theta^{-1} \right)_{mj} .
$$
Hence, the result follows directly:
$$
\left( e_{x_\theta}(u_{\Omega_\theta}) \right)_{ij} = \frac{1}{2} \left( \frac{\partial \left( u_{\Omega_\theta} \right)_i}{\partial \left( x_\theta \right)_j} + \frac{\partial \left( u_{\Omega_\theta} \right)_j}{\partial \left( x_\theta \right)_i} \right) = 
\frac{1}{2} \left( \frac{\partial \left( u_\Omega \right)_i}{\partial \left( x \right)_m} \left( D_\theta^{-1} \right)_{mj} + \frac{\partial \left( u_\Omega \right)_j}{\partial \left( x \right)_m} \left( D_\theta^{-1} \right)_{mi} \right) ,
$$
that is 
$$
e_{x_\theta}(u_{\Omega_\theta}) = \frac{1}{2} \left( \nabla_x u_\Omega D_\theta^{-1} + D_\theta^{-T} \nabla_x u_\Omega^T \right) .
$$
\end{proof}
For the sake of readability and except in the case of ambiguity, henceforth
we will omit the subscript specifying the spatial coordinate with respect to which the 
gradient is computed, that is with an abuse of notation we consider 
$\nabla u_\Omega = \nabla_x u_\Omega$ and 
$\nabla u_{\Omega_\theta} = \nabla_{x_\theta} u_{\Omega_\theta}$.

Now, we use the transformation (\ref{eq:elas_transformationH1}) and the property 
(\ref{eq:elas_linStrain}) to map the first term in (\ref{eq:elas_max_H1}) to 
the reference domain $\Omega$:
\begin{equation}
\begin{aligned}
\int_{\Omega_\theta}{Ae(u_{\Omega_\theta}) : e(u_{\Omega_\theta}) \ dx_\theta} = & \int_\Omega{Ae \left(u_{\Omega_\theta} \circ X_\theta \right) : e \left(u_{\Omega_\theta} \circ X_\theta \right) I_\theta \ dx}  \\
= \int_\Omega A \bigg( \frac{1}{2} \left( \nabla u_\Omega D_\theta^{-1} \right. \hspace{-2pt}+\hspace{-2pt} D_\theta^{-T} & \left. \nabla u_\Omega^T \right) \hspace{-2pt} \bigg) : \left( \frac{1}{2} \left( \nabla u_\Omega D_\theta^{-1} + D_\theta^{-T} \nabla u_\Omega^T \right) \right) I_\theta \ dx .
\end{aligned}
\label{eq:elas_H1_pt1} 
\end{equation}
The remaining terms in (\ref{eq:elas_max_H1}) may be transported to the 
reference domain as follows:
\begin{gather}
- 2 \int_{\Omega_\theta}{ f \cdot u_{\Omega_\theta} \ dx_\theta} = 
- 2 \int_\Omega{ f \circ X_\theta \cdot \left( u_{\Omega_\theta} \circ X_\theta \right) I_\theta \ dx} = 
- 2 \int_\Omega{ f \circ X_\theta \cdot u_\Omega \ I_\theta \ dx} ,
\label{eq:elas_H1_pt2} \\
- 2 \int_{\Gamma^N_\theta}{ g \cdot u_{\Omega_\theta} \ ds_\theta} = 
- 2 \int_{\Gamma^N}{ g \circ X_\theta \cdot \left( u_{\Omega_\theta} \circ X_\theta \right) \Cof D_\theta \ ds} =
- 2 \int_{\Gamma^N}{ g \circ X_\theta \cdot u_\Omega \Cof D_\theta \ ds} ,
\label{eq:elas_H1_pt3}
\end{gather}
where $\Cof D_\theta$ is the cofactor matrix of the jacobian of $X_\theta$.
By combining (\ref{eq:elas_H1_pt1}), (\ref{eq:elas_H1_pt2}) and (\ref{eq:elas_H1_pt3}), 
we obtain the following function $j_1(\theta)$ which solely depends on the reference domain $\Omega$:
\begin{equation}
\begin{aligned}
j_1(\theta) = - \min_{u_\Omega \in V_\Omega} & \int_\Omega{A \left( \frac{1}{2} \left( \nabla u_\Omega D_\theta^{-1} + D_\theta^{-T} \nabla u_\Omega^T \right)  \right) : \left( \frac{1}{2} \left( \nabla u_\Omega D_\theta^{-1} + D_\theta^{-T} \nabla u_\Omega^T \right) \right) I_\theta \ dx} \\
& - 2 \int_\Omega{ f \circ X_\theta \cdot u_\Omega \ I_\theta \ dx} 
- 2 \int_{\Gamma^N}{ g \circ X_\theta \cdot u_\Omega \Cof D_\theta \ ds} .
\end{aligned}
\label{eq:elas_max1}
\end{equation}

Owing to (\ref{eq:elas_perturbId}), the Jacobian of the transformations 
$X_\theta$, $X_\theta^T$ and $X_\theta^{-1}$ read as
\begin{gather}
D_\theta = \Id + \nabla \theta + o(\nabla \theta) , 
\label{eq:elas_jac} \\
D_\theta^T = \Id + \nabla \theta^T + o(\nabla \theta) ,
\label{eq:elas_jacT} \\
D_\theta^{-1} = \Id - \nabla \theta + o(\nabla \theta) .
\label{eq:elas_jacInv}
\end{gather}
Moreover, we recall that
\begin{gather}
\det(\Id + C) = 1 + \tr(C) + o(C) ,
\label{eq:elas_develop_det} \\
\Cof(\Id + C) = \Id + \tr(C)\Id - C + o(C) .
\label{eq:elas_develop_cof} 
\end{gather}

We may now differentiate (\ref{eq:elas_max1}) with respect to $\theta$ in $\theta=0$ by 
exploiting (\ref{eq:elas_jacT}), (\ref{eq:elas_jacInv}), (\ref{eq:elas_develop_det}) and 
(\ref{eq:elas_develop_cof}). The shape gradient of the compliance using the 
pure displacement formulation for the linear elasticity problem reads as
\begin{equation}
\begin{aligned}
\langle dJ_1(\Omega),\theta \rangle = & \int_\Omega{A e(u_\Omega) : \left( \nabla u_\Omega \nabla \theta + \nabla \theta^T \nabla u_\Omega^T \right) \ dx} 
- \int_\Omega{A e(u_\Omega) : e(u_\Omega) (\nabla \cdot \theta) \ dx} \\
& + 2 \int_\Omega{\left( \nabla f \theta \cdot u_\Omega + f \cdot u_\Omega (\nabla \cdot \theta) \right) dx} 
+ 2 \int_{\Gamma^N}{\left( \nabla g \theta \cdot u_\Omega + g \cdot u_\Omega \left( \nabla \cdot \theta - \nabla \theta n \cdot n \right) \right) ds} .
\end{aligned}
\label{eq:elas_shapeGrad_H1}
\end{equation}

\section{Volumetric shape gradient of the compliance via the dual mixed formulation}
\label{ref:shapeGrad_mixed}

Let us consider the notation introduced in section \ref{ref:compliane_minimization} 
for the transformation $X_\theta$.
Following the same procedure as above, we may rewrite the compliance 
coupled with the constraint that the stress tensor is solution of the linear 
elasticity equation in the Hellinger-Reissner dual mixed variational formulation 
(\ref{eq:elas_mixed})-(\ref{eq:elas_dual_a_b})-(\ref{eq:elas_dual_F}) on $\Omega_\theta$
by introducing the following objective functional:
\begin{equation}
J_2(\Omega_\theta) \coloneqq \adjustlimits\inf_{\sigma_{\Omega_\theta} \in \Sigma_{\Omega_\theta}} \sup_{u_{\Omega_\theta} \in V_{\Omega_\theta}} \int_{\Omega_\theta}{A^{-1}\sigma_{\Omega_\theta} : \sigma_{\Omega_\theta} \ dx_\theta} + \int_{\Omega_\theta}{\left(\nabla \cdot \sigma_{\Omega_\theta} + f \right) \cdot u_{\Omega_\theta} \ dx_\theta}
\label{eq:elas_saddle_mixed} 
\end{equation}
where $\Sigma_{\Omega_\theta} \coloneqq \{ \tau \in \hDivT{\mathbb{S}_d} \ : \ \tau n_\theta = g \ \text{on} \ \Gamma^N_\theta \ \text{and} \ \tau n = 0 \ \text{on} \ \Gamma_\theta \}$ 
and $V_{\Omega_\theta} \coloneqq \lTwoT{\mathbb{R}^d}$.
\\
In a similar fashion, 
starting from the dual mixed variational formulation with weakly enforced symmetry of the 
stress tensor (\ref{eq:elas_mixed_weak})-(\ref{eq:elas_dual_weak_a_b_c})-(\ref{eq:elas_dual_weak_F}), 
we obtain:
\begin{equation}
\begin{aligned}
J_3(\Omega_\theta) \coloneqq \adjustlimits\inf_{\sigma_{\Omega_\theta} \in \Sigma_{\Omega_\theta}} \sup_{(u_{\Omega_\theta},\eta_{\Omega_\theta}) \in W_{\Omega_\theta}} & \int_{\Omega_\theta}{A^{-1}\sigma_{\Omega_\theta} : \sigma_{\Omega_\theta} \ dx_\theta} + \int_{\Omega_\theta}{\left(\nabla \cdot \sigma_{\Omega_\theta} + f \right) \cdot u_{\Omega_\theta} \ dx_\theta} \\ 
& + \frac{1}{2\mu} \int_{\Omega_\theta}{\sigma_{\Omega_\theta} : \eta_{\Omega_\theta} \ dx_\theta}
\end{aligned}
\label{eq:elas_saddle_mixed_weak} 
\end{equation}
where $\Sigma_{\Omega_\theta} \coloneqq \{ \tau \in \hDivT{\mathbb{M}_d} \ : \ \tau n_\theta = g \ \text{on} \ \Gamma^N_\theta \ \text{and} \ \tau n = 0 \ \text{on} \ \Gamma_\theta \}$ 
and $W_{\Omega_\theta} \coloneqq V_{\Omega_\theta} \times Q_{\Omega_\theta} \coloneqq \lTwoT{\mathbb{R}^d} \times \lTwoT{\mathbb{K}_d}$.
\\
Let $j_i(\theta) \coloneqq J_i(\Omega_\theta) \ i=2,3$. We are interested in computing 
the shape gradient of the functionals $J_i(\Omega)$'s, that is
\begin{equation}
\langle dJ_i(\Omega),\theta \rangle \coloneqq \lim_{\theta \searrow 0} \frac{J_i(\Omega_\theta)-J_i(\Omega)}{\theta} = \lim_{\theta \searrow 0} \frac{j_i(\theta)-j_i(0)}{\theta} \eqqcolon j_i'(0) .
\label{eq:elas_minmaxGrad}
\end{equation}
We refer to \cite{MR776359} for a general result on the differentiability of a min-max function, 
whereas in \cite{Delfour:2001:SGA:501610, MR2431476} some examples of shape differentiability of 
min-max functions are provided. \\
As in section \ref{ref:shapeGrad_H1}, we apply the function space parameterization technique 
to transport the quantities defined on $\Omega_\theta$ back to $\Omega$.
A key aspect of this procedure is the construction of a transformation that preserves the 
normal traces of the tensors in (\ref{eq:elas_saddle_mixed}) and (\ref{eq:elas_saddle_mixed_weak}). 
For this purpose, we rely on a special isomorphism known as contravariant Piola transform and 
we define the following mappings:
\begin{gather}
\mathcal{Q}_\theta : \hDiv{\mathbb{M}_d} \rightarrow \hDivT{\mathbb{M}_d} \quad , \quad \tau_{\Omega_\theta} = \mathcal{Q}_\theta(\tau_\Omega) = \frac{1}{I_\theta}D_\theta \tau_\Omega \circ X_\theta^{-1} D_\theta^T 
\label{eq:elas_transformationHdiv} \\
\mathcal{R}_\theta : \lTwo{\mathbb{R}^d} \rightarrow \lTwoT{\mathbb{R}^d} \quad , \quad v_{\Omega_\theta} = \mathcal{R}_\theta(v_\Omega) = D_\theta^{-T} v_\Omega \circ X_\theta^{-1} .
\label{eq:elas_transformationL2}
\end{gather}
We refer to \cite{MR1262126, MR936420} for a discussion on the Piola transform and its 
role in the mathematical theory of elasticity, to \cite{Thomas1976, MR0483555} for its 
application to mixed Finite Element methods for elliptic problems and to \cite{MR2566587} 
for some technical details on its use to efficiently evaluate variational forms in 
$\Hdiv$ and $\Hcurl$, that is the Sobolev space of square-integrable vectorfields 
whose rotation $\operatorname{curl}$ is square-integrable.
\\
Before moving to the derivation of the shape gradient via the function space parameterization 
technique, we prove the following property:
\begin{lemma}
Let $\sigma_\Omega \in \hDiv{\mathbb{M}_d}$.
We consider $\sigma_{\Omega_\theta} = \mathcal{Q}_\theta(\sigma_\Omega)$ according to the transformation 
(\ref{eq:elas_transformationHdiv}). It follows that
\begin{equation}
\nabla_{x_\theta} \cdot \sigma_{\Omega_\theta} = \frac{1}{I_\theta} D_\theta \nabla_x \cdot \sigma_\Omega
\label{eq:elas_divPiola}
\end{equation}
where $\nabla_{x_\theta} \cdot$ (respectively $\nabla_x \cdot$) represents the divergence with respect to 
the coordinate of the deformed (respectively reference) domain.
\label{theo:divPiola}
\end{lemma}
\begin{proof}
First, we recall that for a given invertible matrix $C \in \mathbb{M}_d$, we get that 
\begin{equation}
C^{-1} = \frac{1}{\det C} (\Cof C)^T .
\label{eq:invCof}
\end{equation}
Owing to this property, we may rewrite (\ref{eq:elas_transformationHdiv}) as
\begin{equation}
\sigma_{\Omega_\theta} = D_\theta \sigma_\Omega \circ X_\theta^{-1} \left( \Cof D_\theta^{-1} \right) . 
\label{eq:elas_newTheta}
\end{equation}
We are interested in computing the divergence of (\ref{eq:elas_newTheta}) with respect 
to the coordinate $x_\theta$ of the deformed domain. Within this framework, we observe that 
being $D_\theta$ the Jacobian of the transformation (\ref{eq:elas_transformation}) such 
that $\Omega_\theta \ni x_\theta = X_\theta(x) \ , \ x \in \Omega$, it is independent on 
the variable $x_\theta$. Let us now prove the following Piola identity:
\begin{equation}
\nabla_{x_\theta} \cdot \left( \Cof D_\theta^{-1} \right) = 0 .
\label{eq:elas_piolaId} 
\end{equation}
Using the Levi-Civita symbol $\varepsilon_{ijk}$ and the Einstein summation convention, 
the cofactor matrix of the inverse of the Jacobian $D_\theta$ has the form
$$
\left( \Cof D_\theta^{-1} \right)_{ij} = \frac{1}{2} \varepsilon_{imn} \varepsilon_{jpq} \frac{\partial (X_\theta^{-1})_m}{\partial (x_\theta)_p} \frac{\partial (X_\theta^{-1})_n}{\partial (x_\theta)_q} .
$$
Its divergence reads
$$
\begin{aligned}
\frac{\partial \left( \Cof D_\theta^{-1} \right)_{ij}}{\partial \left(x_\theta\right)_j} & = \frac{1}{2} \varepsilon_{imn} \varepsilon_{jpq} \left( \frac{\partial^2 \left(X_\theta^{-1}\right)_m}{\partial \left(x_\theta\right)_j \partial \left(x_\theta\right)_p} \frac{\partial \left(X_\theta^{-1}\right)_n}{\partial \left(x_\theta\right)_q} + \frac{\partial \left(X_\theta^{-1}\right)_m}{\partial \left(x_\theta\right)_p} \frac{\partial^2 \left(X_\theta^{-1}\right)_n}{\partial \left(x_\theta\right)_j \partial \left(x_\theta\right)_q} \right) \\
& = \frac{1}{2} \varepsilon_{imn} \left( \varepsilon_{pjq} \frac{\partial^2 \left(X_\theta^{-1}\right)_m}{\partial \left(x_\theta\right)_p \partial \left(x_\theta\right)_j} \frac{\partial \left(X_\theta^{-1}\right)_n}{\partial \left(x_\theta\right)_q} + \varepsilon_{qpj} \frac{\partial \left(X_\theta^{-1}\right)_m}{\partial \left(x_\theta\right)_p} \frac{\partial^2 \left(X_\theta^{-1}\right)_n}{\partial \left(x_\theta\right)_q \partial \left(x_\theta\right)_j} \right) \\
& = - \frac{1}{2} \varepsilon_{imn} \varepsilon_{jpq} \left( \frac{\partial^2 \left(X_\theta^{-1}\right)_m}{\partial \left(x_\theta\right)_j \partial \left(x_\theta\right)_p} \frac{\partial \left(X_\theta^{-1}\right)_n}{\partial \left(x_\theta\right)_q} + \frac{\partial \left(X_\theta^{-1}\right)_m}{\partial \left(x_\theta\right)_p} \frac{\partial^2 \left(X_\theta^{-1}\right)_n}{\partial \left(x_\theta\right)_j \partial \left(x_\theta\right)_q} \right) \\
& = - \frac{\partial \left( \Cof D_\theta^{-1} \right)_{ij}}{\partial \left(x_\theta\right)_j} ,
\end{aligned}
$$
where the third equality follows from the definition of the Levi-Civita 
symbol. Hence, we can conclude that (\ref{eq:elas_piolaId}) stands.
We may now compute the divergence of (\ref{eq:elas_newTheta}): 
$$
\begin{aligned}
\nabla_{x_\theta} \cdot \sigma_{\Omega_\theta} = \frac{\partial \left(\sigma_{\Omega_\theta}\right)_{ij}}{\partial \left(x_\theta\right)_j} e_i & = \frac{\partial}{\partial \left(x_\theta\right)_j} \left( \left( D_\theta \right)_{im} \left(\sigma_\Omega \circ X_\theta^{-1} \right)_{mq} \left( \Cof D_\theta^{-1} \right)_{qj} \right) e_i \\
& = \left( D_\theta \right)_{im} \frac{\partial \left(\sigma_\Omega\right)_{mn}}{\partial \left(x\right)_n} \frac{\partial \left(X_\theta^{-1}\right)_n}{\partial \left(x_\theta\right)_j} \left( \Cof D_\theta^{-1} \right)_{qj} e_i \\
& = \left( D_\theta \right)_{im} \frac{\partial \left(\sigma_\Omega\right)_{mn}}{\partial \left(x\right)_n} \left(D_\theta^{-1}\right)_{nj} \left( \Cof D_\theta^{-1} \right)_{qj} e_i \\
& = \frac{1}{\det D_\theta} \left( D_\theta \right)_{im} \frac{\partial \left(\sigma_\Omega\right)_{mn}}{\partial \left(x\right)_n} \delta_{nq} e_i
\end{aligned}
$$
where the last equality follows from (\ref{eq:invCof}).
Hence, it is straightforward to retrieve the result (\ref{eq:elas_divPiola}):
$$
\nabla_{x_\theta} \cdot \sigma_{\Omega_\theta} = \frac{1}{I_\theta} \left( D_\theta \right)_{im} \frac{\partial \left(\sigma_\Omega\right)_{mq}}{\partial \left(x\right)_q} e_i  = \frac{1}{I_\theta} D_\theta \nabla_x \cdot \sigma_\Omega
$$
\end{proof}
From now on, if there is no ambiguity 
we will assume that the differential operators act on the space to which the functions belong
and we will omit the subscript associated with the spatial coordinate used to compute the 
derivatives (e.g. $\nabla \cdot \sigma_{\Omega_\theta} = \nabla_{x_\theta} \cdot \sigma_{\Omega_\theta}$ 
and $\nabla \cdot \sigma_\Omega = \nabla_x \cdot \sigma_\Omega$).

As stated at the beginning of this section, in order to compute the shape gradients 
(\ref{eq:elas_minmaxGrad}), we have to express the functionals $J_2(\Omega_\theta)$ and 
$J_3(\Omega_\theta)$ in terms of the reference domain $\Omega$ and of functions defined 
solely on it.
Thus, in the following subsections we use the transformations (\ref{eq:elas_transformationHdiv}) 
and (\ref{eq:elas_transformationL2}) to map (\ref{eq:elas_saddle_mixed}) and 
(\ref{eq:elas_saddle_mixed_weak}) back to the reference domain and differentiate them 
with respect to $\theta$.

\subsection{The case of strongly enforced symmetry of the stress tensor}
\label{ref:strong_sym}

We consider the Hellinger-Reissner mixed variational formulation of the linear elasticity 
problem and the corresponding objective functional (\ref{eq:elas_saddle_mixed}).
We remark that the symmetry of the stress tensor $\sigma_{\Omega_\theta}$ is 
strongly enforced using the space 
$\Sigma_{\Omega_\theta}\coloneqq \{ \tau \in \hDivT{\mathbb{S}_d} \ : \ \tau n_\theta = g \ \text{on} \ \Gamma^N_\theta \ \text{and} \ \tau n = 0 \ \text{on} \ \Gamma_\theta \}$.
It is straightforward to observe that the transformation (\ref{eq:elas_transformationHdiv}) 
holds true for the space of $d \times d$ symmetric matrices $\mathbb{S}_d$, that is 
$\mathcal{Q}_\theta : \hDiv{\mathbb{S}_d} \rightarrow \hDivT{\mathbb{S}_d}$.
As a matter of fact, being $\tau_\Omega \in \hDiv{\mathbb{S}_d}$, it follows that 
$$
\left( \tau_{\Omega_\theta} \right)^T = \left( \frac{1}{I_\theta}D_\theta \tau_\Omega \circ X_\theta^{-1} D_\theta^T \right)^T = \frac{1}{I_\theta}D_\theta \tau_\Omega \circ X_\theta^{-1} D_\theta^T = \tau_{\Omega_\theta} .
$$

We use the definition of the compliance tensor in (\ref{eq:elas_hooke_inverse}) and we 
map the first term in (\ref{eq:elas_saddle_mixed}) to the reference domain $\Omega$ 
by means of the transformation (\ref{eq:elas_transformationHdiv}):
\begin{equation}
\begin{aligned}
\int_{\Omega_\theta}{\sigma_{\Omega_\theta} : \sigma_{\Omega_\theta} \ dx_\theta} & = \int_\Omega{\left( \sigma_{\Omega_\theta} \circ X_\theta \right) : \left( \sigma_{\Omega_\theta} \circ X_\theta \right) I_\theta \ dx} \\
& = \int_\Omega{\frac{1}{I_\theta^2} \left( D_\theta \sigma_\Omega D_\theta^T \right) : \left( D_\theta \sigma_\Omega D_\theta^T \right) I_\theta \ dx} \\
& = \int_\Omega{\frac{1}{I_\theta} D_\theta^T D_\theta \sigma_\Omega D_\theta^T D_\theta : \sigma_\Omega \ dx} ,
\end{aligned}
\label{eq:elas_hr_pt1}
\end{equation}
where the last equality follows from the definition of the Frobenius product and the cyclic 
property of the trace.
In a similar fashion, we obtain
\begin{equation}
\begin{aligned}
\int_{\Omega_\theta}{\tr(\sigma_{\Omega_\theta})\tr(\sigma_{\Omega_\theta}) \ dx_\theta} & = \int_\Omega{\tr\left( \sigma_{\Omega_\theta} \circ X_\theta \right)\tr\left( \sigma_{\Omega_\theta} \circ X_\theta \right) I_\theta \ dx} \\
& = \int_\Omega{\frac{1}{I_\theta^2} \tr\left( D_\theta \sigma_\Omega D_\theta^T \right) \tr\left( D_\theta \sigma_\Omega D_\theta^T \right) I_\theta \ dx} \\
& = \int_\Omega{\frac{1}{I_\theta} \tr\left( D_\theta^T D_\theta \sigma_\Omega \right) \tr\left( D_\theta^T D_\theta \sigma_\Omega \right) dx} .
\end{aligned}
\label{eq:elas_hr_pt2}
\end{equation}
We consider now the second term in (\ref{eq:elas_saddle_mixed}). 
Owing to (\ref{eq:elas_divPiola}) and (\ref{eq:elas_transformationL2}) it follows
\begin{equation}
\begin{aligned}
\int_{\Omega_\theta}{\left(\nabla \cdot \sigma_{\Omega_\theta} \right) \cdot u_{\Omega_\theta} \ dx_\theta} & = \int_\Omega{\left(\nabla \cdot \left( \sigma_{\Omega_\theta} \circ X_\theta \right) \right) \cdot \left( u_{\Omega_\theta} \circ X_\theta \right) I_\theta \ dx} \\
& = \int_\Omega{\frac{1}{I_\theta} \left( D_\theta \nabla \cdot \sigma_\Omega \right) \cdot \left( D_\theta^{-T} u_\Omega \right) I_\theta \ dx} 
= \int_\Omega{\left( \nabla \cdot \sigma_\Omega \right) \cdot u_\Omega \ dx} ,
\end{aligned}
\label{eq:elas_hr_pt3} 
\end{equation}
\begin{equation}
\int_{\Omega_\theta}{ f \cdot u_{\Omega_\theta} \ dx_\theta} = 
\int_\Omega{ f \circ X_\theta \cdot \left( u_{\Omega_\theta} \circ X_\theta \right) I_\theta \ dx} = 
\int_\Omega{ f \circ X_\theta \cdot \left( D_\theta^{-T} u_\Omega \right) I_\theta \ dx} .
\label{eq:elas_hr_pt4}
\end{equation}

By combining the above information, 
we obtain the following min-max function which no longer depends on the space $\Omega_\theta$:
\begin{equation}
\begin{aligned}
j_2(\theta) = \adjustlimits\inf_{\sigma_\Omega \in \Sigma_\Omega} \sup_{u_\Omega \in V_\Omega} \frac{1}{2\mu} & \int_\Omega{\frac{1}{I_\theta} D_\theta^T D_\theta \sigma_\Omega D_\theta^T D_\theta : \sigma_\Omega \ dx} \\
& - \frac{\lambda}{2\mu(d\lambda+2\mu)} \int_\Omega{\frac{1}{I_\theta} \tr\left( D_\theta^T D_\theta \sigma_\Omega \right) \tr\left( D_\theta^T D_\theta \sigma_\Omega \right) dx} \\
& + \int_\Omega{\left( \nabla \cdot \sigma_\Omega \right) \cdot u_\Omega \ dx} + \int_\Omega{ f \circ X_\theta \cdot \left( D_\theta^{-T} u_\Omega \right) I_\theta \ dx} .
\end{aligned}
\label{eq:elas_minimax1}
\end{equation}
We may now exploit (\ref{eq:elas_jac}), (\ref{eq:elas_jacT}) and (\ref{eq:elas_develop_det})
to differentiate (\ref{eq:elas_minimax1}) with respect to $\theta$ and evaluate 
the resulting quantity in $\theta=0$. Thus, the shape gradient of the compliance 
using the Hellinger-Reissner dual mixed variational formulation for the linear elasticity 
problem reads as
\begin{equation}
\begin{aligned}
\langle dJ_2(\Omega),\theta \rangle = & \frac{1}{\mu} \int_\Omega{N(\theta) \sigma_\Omega : \sigma_\Omega \ dx} 
- \frac{\lambda}{\mu(d\lambda+2\mu)} \int_\Omega{\tr\left( N(\theta) \sigma_\Omega \right) \tr\left( \sigma_\Omega \right) dx} \\
& + \int_\Omega{\left( \nabla f \theta \cdot u_\Omega + f \cdot u_\Omega (\nabla \cdot \theta) - f \cdot (\nabla \theta^T u_\Omega) \right) dx} 
\end{aligned}
\label{eq:elas_shapeGrad_mixed1}
\end{equation}
where $N(\theta) \coloneqq \nabla \theta + \nabla \theta^T - \frac{1}{2} (\nabla \cdot \theta) \Id$.

\subsection{The case of weakly enforced symmetry of the stress tensor}
\label{ref:weak_sym}

The dual mixed formulation of the linear elasticity problem discussed in subsection 
\ref{ref:mixed_weak} is characterized by the weak imposition of the symmetry of 
the stress tensor through a Lagrange multiplier $\eta_{\Omega_\theta}$.
Thus, besides the spaces $V_{\Omega_\theta}$ and $\Sigma_{\Omega_\theta}$, the functional 
(\ref{eq:elas_saddle_mixed_weak}) associated with the minimization of the compliance 
using the aforementioned framework introduces the additional space 
$Q_{\Omega_\theta} \coloneqq \lTwoT{\mathbb{K}_d}$ of the $d \times d$ skew-symmetric square-integrable tensors.
In order to map the space $\lTwo{\mathbb{K}_d}$ to $\lTwoT{\mathbb{K}_d}$, we 
use the previously introduced transformation (\ref{eq:elas_transformationHdiv}):
it is straightforward to observe that given $\eta_\Omega \in \lTwo{\mathbb{K}_d}$, the transported 
$\eta_{\Omega_\theta} = \mathcal{Q}_\theta(\eta_\Omega)$ is skew-symmetric:
$$
\left( \eta_{\Omega_\theta} \right)^T = \left( \frac{1}{I_\theta}D_\theta \eta_\Omega \circ X_\theta^{-1} D_\theta^T \right)^T = \frac{1}{I_\theta}D_\theta \left( \eta_\Omega \circ X_\theta^{-1} \right)^T D_\theta^T = - \frac{1}{I_\theta}D_\theta \eta_\Omega \circ X_\theta^{-1} D_\theta^T = - \eta_{\Omega_\theta} .
$$

The first two integrals in (\ref{eq:elas_saddle_mixed_weak}) may be treated as in the previous 
subsection and the manipulations that lead to (\ref{eq:elas_hr_pt1}), (\ref{eq:elas_hr_pt2}), (\ref{eq:elas_hr_pt3}) 
and (\ref{eq:elas_hr_pt4}) stand. Let us now map the remaining term in (\ref{eq:elas_saddle_mixed_weak}) 
back to the reference domain $\Omega$:
\begin{equation}
\begin{aligned}
\int_{\Omega_\theta}{\sigma_{\Omega_\theta} : \eta_{\Omega_\theta} \ dx_\theta} & = \int_\Omega{\left( \sigma_{\Omega_\theta} \circ X_\theta \right) : \left( \eta_{\Omega_\theta} \circ X_\theta \right) I_\theta \ dx} \\
& = \int_\Omega{\frac{1}{I_\theta^2} \left( D_\theta \sigma_\Omega D_\theta^T \right) : \left( D_\theta \eta_\Omega D_\theta^T \right) I_\theta \ dx} \\
& = \int_\Omega{\frac{1}{I_\theta} D_\theta^T D_\theta \sigma_\Omega D_\theta^T D_\theta : \eta_\Omega \ dx} .
\end{aligned}
\label{eq:elas_weak_pt5}
\end{equation}
We combine (\ref{eq:elas_hr_pt1}), (\ref{eq:elas_hr_pt2}), (\ref{eq:elas_hr_pt3}), (\ref{eq:elas_hr_pt4}) 
and (\ref{eq:elas_weak_pt5}) to obtain the min-max function associated with $J_3(\Omega_\theta)$  
and defined on a space that does not depend on $\theta$:
\begin{equation}
\begin{aligned}
j_3(\theta) = \adjustlimits\inf_{\sigma_\Omega \in \Sigma_\Omega} \sup_{(u_\Omega,\eta_\Omega) \in W_\Omega} \frac{1}{2\mu} & \int_\Omega{\frac{1}{I_\theta} D_\theta^T D_\theta \sigma_\Omega D_\theta^T D_\theta : \sigma_\Omega \ dx} \\
& - \frac{\lambda}{2\mu(d\lambda+2\mu)} \int_\Omega{\frac{1}{I_\theta} \tr\left( D_\theta^T D_\theta \sigma_\Omega \right) \tr\left( D_\theta^T D_\theta \sigma_\Omega \right) dx} \\
& + \frac{1}{2\mu} \int_\Omega{\frac{1}{I_\theta} D_\theta^T D_\theta \sigma_\Omega D_\theta^T D_\theta : \eta_\Omega \ dx} \\
& + \int_\Omega{\left( \nabla \cdot \sigma_\Omega \right) \cdot u_\Omega \ dx} + \int_\Omega{ f \circ X_\theta \cdot \left( D_\theta^{-T} u_\Omega \right) I_\theta \ dx} .
\end{aligned}
\label{eq:elas_minimax2}
\end{equation}
Let us consider the matrix $N(\theta)$ introduced in the previous subsection.
By differentiating (\ref{eq:elas_minimax2}) with respect to $\theta$ in $\theta=0$, we obtain 
the following expression of the shape gradient of the compliance using the dual mixed variational 
formulation for the linear elasticity with weakly imposed symmetry of the stress tensor:
\begin{equation}
\begin{aligned}
\langle dJ_3(\Omega),\theta \rangle = & \frac{1}{2\mu} \int_\Omega{\left( N(\theta) \sigma_\Omega : \sigma_\Omega + \sigma_\Omega N(\theta) : \sigma_\Omega \right) dx} \\
& - \frac{\lambda}{\mu(d\lambda+2\mu)} \int_\Omega{\tr\left( N(\theta) \sigma_\Omega \right) \tr\left( \sigma_\Omega \right) dx} \\
& + \frac{1}{2\mu} \int_\Omega{\left( N(\theta) \sigma_\Omega : \eta_\Omega + \sigma_\Omega N(\theta) : \eta_\Omega \right) dx} \\
& + \int_\Omega{\left( \nabla f \theta \cdot u_\Omega + f \cdot u_\Omega (\nabla \cdot \theta) - f \cdot (\nabla \theta^T u_\Omega) \right) dx} .
\end{aligned}
\label{eq:elas_shapeGrad_mixed2}
\end{equation}

We remark that the two expressions of the shape gradient obtained 
using the dual mixed variational formulations in subsections \ref{ref:strong_sym} 
and \ref{ref:weak_sym} are equivalent:
\begin{lemma}
Let us consider a symmetric stress tensor $\sigma_\Omega \in \hDiv{\mathbb{S}_d}$.
Then (\ref{eq:elas_shapeGrad_mixed1}) and (\ref{eq:elas_shapeGrad_mixed2}) are equal.
\end{lemma}
\begin{proof}
It is straightforward to observe that the second and the fourth integrals in 
(\ref{eq:elas_shapeGrad_mixed2}) correspond to the last two terms in (\ref{eq:elas_shapeGrad_mixed1}).
Moreover, owing to the symmetry of $N(\theta)$ and $\sigma_\Omega$, we get:
$$
\begin{aligned}
\int_\Omega{\left( N(\theta) \sigma_\Omega : \sigma_\Omega + \sigma_\Omega N(\theta) : \sigma_\Omega \right) dx} 
& = \int_\Omega{\left( \tr\left(N(\theta) \sigma_\Omega \sigma_\Omega^T \right) + \tr\left(N(\theta)^T \sigma_\Omega^T \sigma_\Omega \right) \right) dx} \\
& = \int_\Omega{ 2 \tr\left(N(\theta) \sigma_\Omega \sigma_\Omega^T \right) dx} = \int_\Omega{ 2 N(\theta) \sigma_\Omega : \sigma_\Omega \ dx} .
\end{aligned}
$$
In order to prove the equality $\langle dJ_2(\Omega),\theta \rangle = \langle dJ_3(\Omega),\theta \rangle$, 
we have to show that the following quantity is equal to zero:
$$
\int_\Omega{\left( N(\theta) \sigma_\Omega : \eta_\Omega + \sigma_\Omega N(\theta) : \eta_\Omega \right) dx} = 
\int_\Omega{\left( \tr\left( N(\theta) \sigma_\Omega \eta_\Omega^T \right) + \tr\left( N(\theta)^T \sigma_\Omega^T \eta_\Omega \right) \right) dx} .
$$
The result follows directly from the symmetry of the matrix $N(\theta)$, 
the symmetry of $\sigma_\Omega$ and the skew-symmetry of $\eta_\Omega$.
\end{proof}

\section{Qualitative assessment of the discretized shape gradients via numerical simulations}
\label{ref:shapeGrad_comparison}

In this section, we provide some numerical simulations to present a preliminary comparison of the 
expressions of the shape gradient of the compliance derived 
using different formulations of the linear elasticity problem. 
As mentioned in subsection \ref{ref:mixed}, a major drawback of the Hellinger-Reissner 
variational formulation for the linear elasticity equation is the complexity of the 
stable Arnold-Winther pair of Finite Element spaces associated with this discretization 
(cf. \cite{MR1930384}).
Hence, for the scope of this section, we restrict ourselves to the expression 
of the shape gradient obtained by the pure displacement 
formulation (cf. sections \ref{ref:pureDisp} and \ref{ref:shapeGrad_H1}) 
and to the one arising from the dual mixed formulation with weakly imposed symmetry of 
the stress tensor (cf. sections \ref{ref:mixed_weak} and \ref{ref:weak_sym}).
\\
We consider the optimal design of the classical cantilever beam described in figure 
\ref{fig:cantilever-scheme}.
In particular, we assume a zero body forces configuration, a structure clamped on 
$\Gamma^D$, with a load $g=(0,-1)$ applied on $\Gamma^N$ and a free boundary $\Gamma$.

\begin{figure}[htbp]
\centering
\includegraphics[width=0.5\columnwidth]{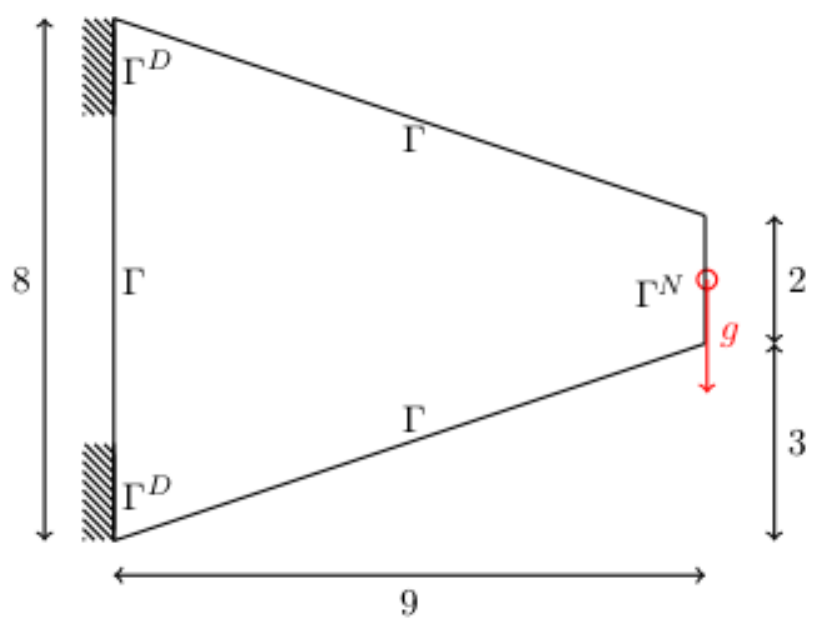}
\caption{Scheme of a 2D cantilever beam clamped on $\Gamma^D$, with a load $g$ applied on the boundary $\Gamma^N$ and free boundaries $\Gamma$.}
\label{fig:cantilever-scheme}
\end{figure}

\begin{figure}[htbp]
    \subfloat[Bulky structure.]
    {
    \begin{minipage}{0.38\columnwidth}
      \includegraphics[width=\columnwidth]{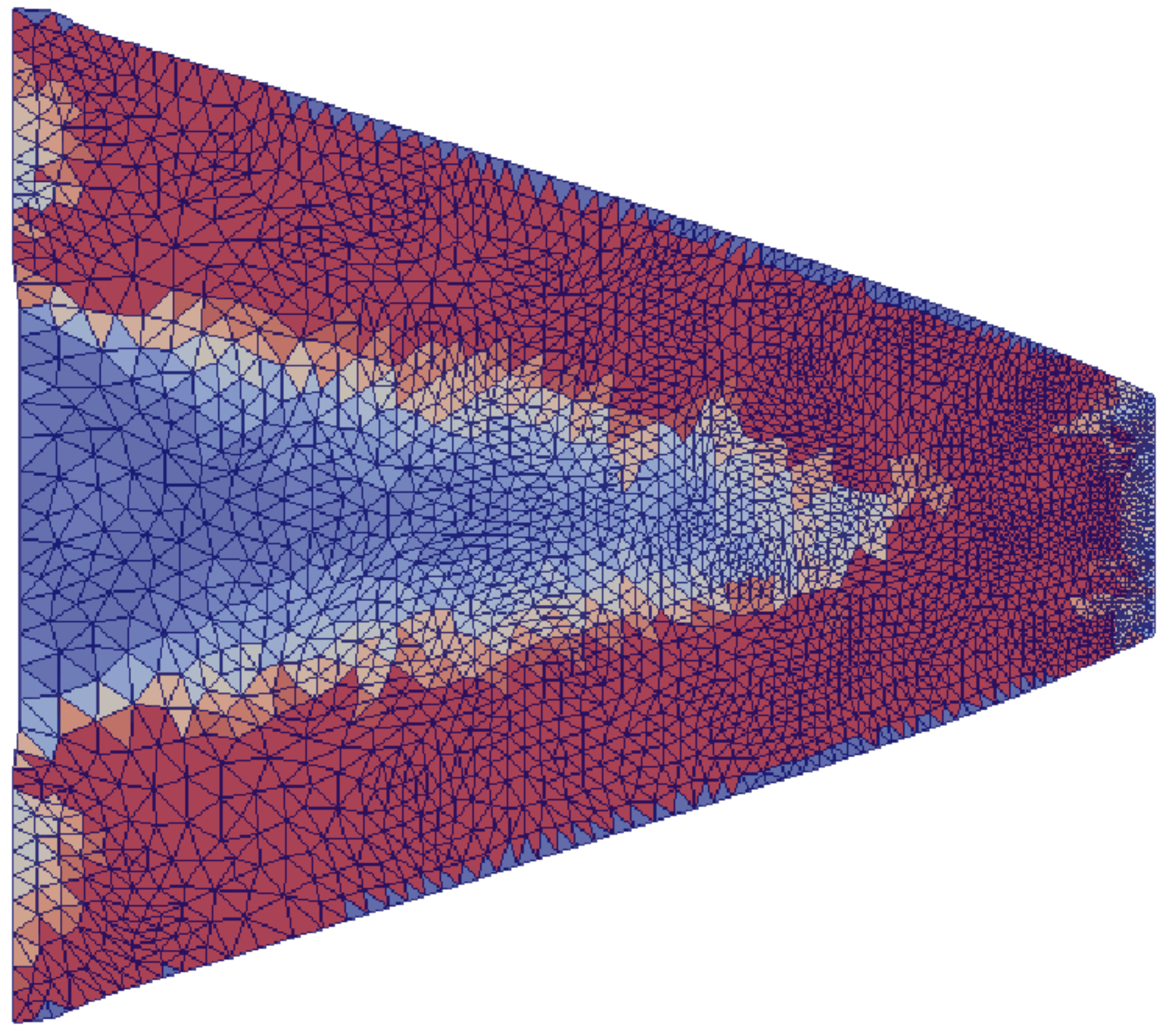}
    \end{minipage}
    \begin{minipage}{0.1\columnwidth}
      \includegraphics[width=\columnwidth]{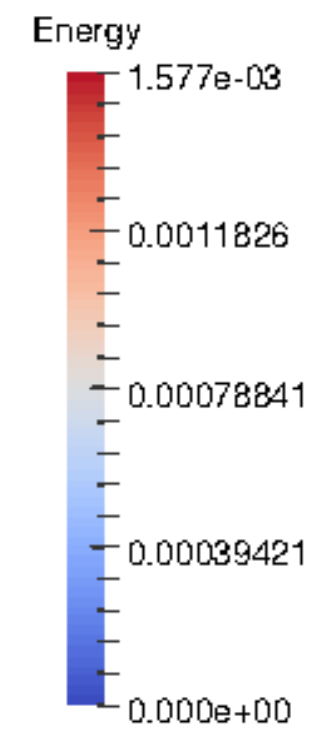}
    \end{minipage}
    \label{fig:initialNO}
    }
    \hfil
    \subfloat[Structure with six holes.]
    {
    \begin{minipage}{0.38\columnwidth}    
      \includegraphics[width=\columnwidth]{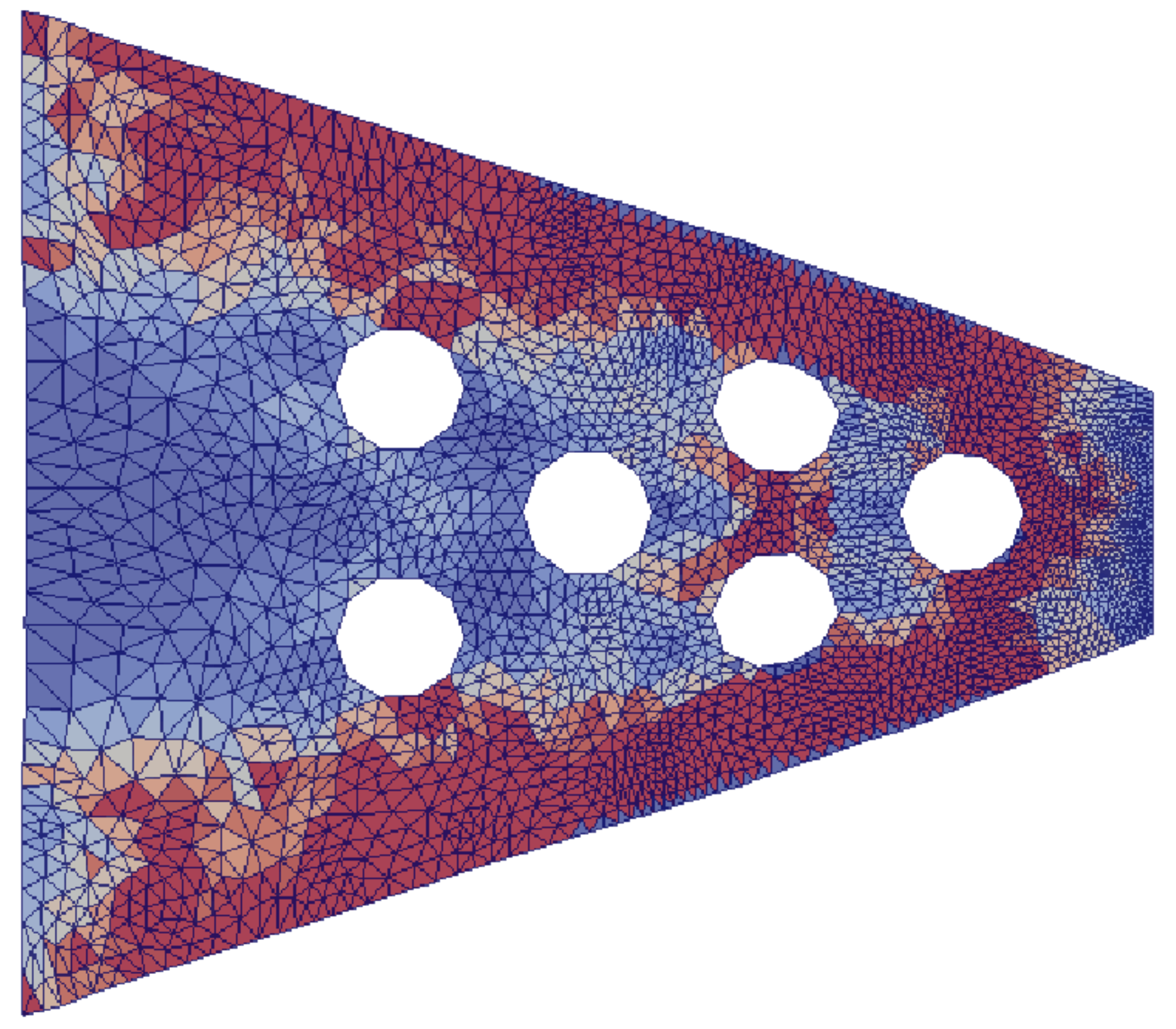}
    \end{minipage}
    \begin{minipage}{0.1\columnwidth}
      \includegraphics[width=\columnwidth]{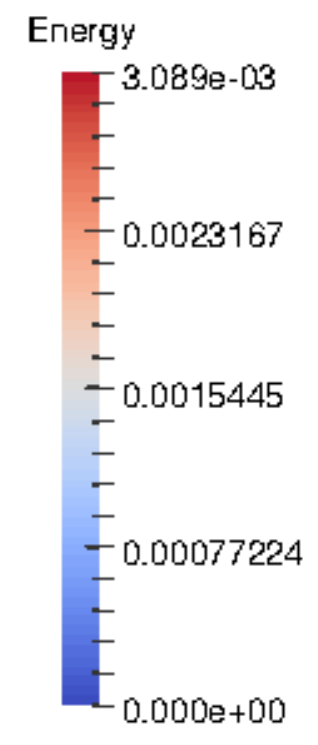}
    \end{minipage}
    \label{fig:initial6}
    }

    \caption{Initial shape and computational mesh for (a) a bulky cantilever and 
    (b) a structure featuring six holes. Density distribution of the elastic energy 
    within the range (a) $(0,1.5 \cdot 10^{-3})$ and 
    (b) $(0,3 \cdot 10^{-3})$, the lower values being in blue and 
    the higher ones in red.}
    \label{fig:initialCanti}
\end{figure}

\begin{figure}[hbtp]
\centering
\includegraphics[width=0.4\columnwidth]{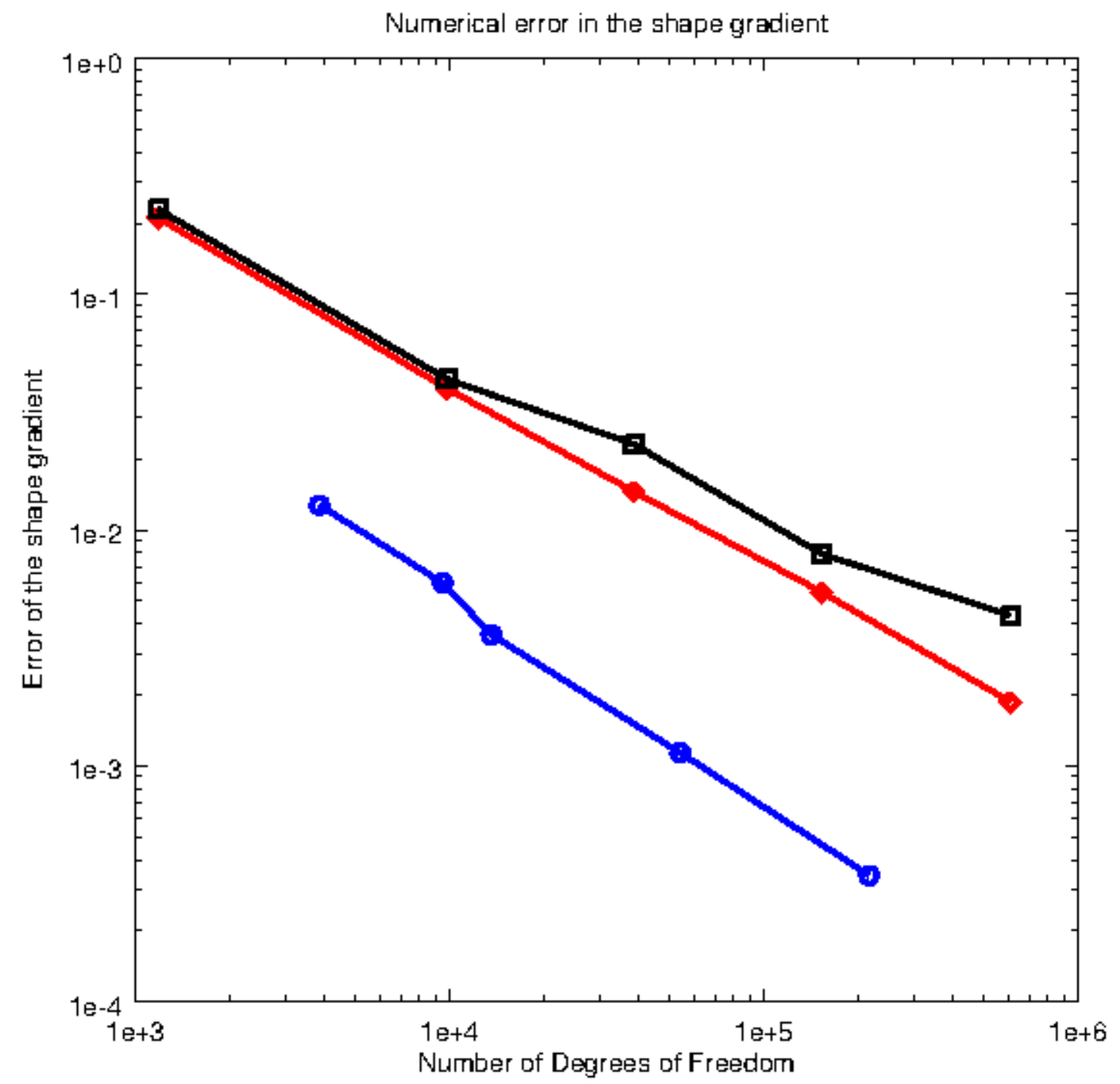}
\caption{Experimental convergence rate of the error in the shape gradient 
computed using the surface expression based on the pure displacement formulation 
(black squares), the corresponding volumetric expression (red diamond) 
and the dual mixed formulation (blue circle) with respect to the 
number of Degrees of Freedom.}
\label{fig:convShapeGrad}
\end{figure}

\subsection{Experimental analysis of the convergence of the error in the shape gradient}

In order to establish an experimental convergence rate for the discretization error 
associated with the approximation of the pure displacement and the dual mixed formulations 
of the linear elasticity problem, we consider the cantilever beam described in figure 
\ref{fig:cantilever-scheme}. In particular, we consider the domain featuring six 
holes depicted in figure \ref{fig:initial6}.
Owing to the fact that the analytical solution of the linear elasticity problem on the 
aforementioned domain $\Omega$ is not known, we solve the linear elasticity problem on an 
extremely fine mesh and we consider the resulting solution as the exact solution of the problem 
under analysis.
The discretization of the pure displacement formulation of the state problem is performed 
using $\mathbb{P}^1 \times \mathbb{P}^1$ Finite Element functions to approximate the 
displacement field.
For the dual mixed formulation, we consider the scheme described in subsection 
\ref{ref:mixed_weak} and we approximate the stress tensor using $BDM_1 \times BDM_1$ 
Finite Elements, the displacement field via $\mathbb{P}^0 \times \mathbb{P}^0$ and 
the Lagrange multiplier by means of a $\mathbb{P}^0$ function.
\\
In figure \ref{fig:convShapeGrad}, we present the convergence history of the 
discretization error in the shape gradient with respect to the number of Degrees 
of Freedom using the surface expression based on the pure displacement formulation 
and the volumetric expressions previously derived.
In particular, we observe that under uniform mesh refinements the surface expression 
based on the pure displacement formulation is less accurate and presents a slower 
convergence rate than the corresponding volumetric one.
Moreover, using the dual mixed formulation the numerical error in the shape gradient 
is furtherly lowered and the blue curve seems slightly steeper than the red one. 
Thus, from the numerical experiments it seems that the volumetric shape gradient 
obtained from the dual mixed formulation of the problem may provide better 
convergence rate than the corresponding expression based on the pure displacement formulation.
Nevertheless, this conjecture remains to be proved and a rigorous analysis 
of the convergence rate by means of \emph{a priori} estimates of the 
error in the shape gradient is necessary.

\subsection{Boundary Variation Algorithm using the pure displacement and the dual mixed formulations}

In this subsection, we apply the Boundary Variation Algorithm described in subsection \ref{ref:BVA} to minimize the 
compliance of the cantilever in figure \ref{fig:cantilever-scheme} under a volume constraint. 
In particular, the volume of the structure under analysis is set to its initial value $V_0$ and we aim 
to construct an optimal shape that minimizes the compliance while preserving as much as possible 
the value $V_0$ of the volume.
As discussed in section \ref{ref:compliane_minimization}, the volume constraint is handled through 
a Lagrange multiplier $\gamma$.
From a theoretical point of view, the value of the Lagrange multiplier should be updated at each 
iteration in order for the optimal shape to fulfill the volume constraint when the algorithm converges.
Nevertheless, enforcing the volume constraint at each iteration would highly increase the complexity of 
the algorithm and consequently its computational cost. Thus we consider a constant Lagrange multiplier 
at each iteration of the strategy and starting from the previously computed value $\gamma$, we increase 
it if the current volume $V$ is greater than the target $V_0$ and we decrease it otherwise.

As extensively discussed in \cite{1742-6596-657-1-012004, giacomini:hal-01201914, giacomini:flux}, 
a key aspect of shape optimization procedures is the choice of the criterion to stop the 
evolution of the optimization strategy. 
In order to compare the expressions (\ref{eq:elas_shapeGrad_H1}) and (\ref{eq:elas_shapeGrad_mixed2}) 
of the shape gradient of the compliance, we consider an \emph{a priori} fixed number of 
iterations for the BVA under analysis.
Moreover, the number of connected regions inside the domain is set at the beginning of the procedure 
and the deformation of the shape is performed via a moving mesh approach.
In the rest of this subsection, we present two test cases for the optimal design of the cantilever 
in figure \ref{fig:cantilever-scheme}, that is a bulky structure (Fig. \ref{fig:initialNO}) and a 
porous one featuring six internal holes (Fig. \ref{fig:initial6}).
All the numerical simulations are obtained using FreeFem++ \cite{MR3043640}.

\subsubsection*{Bulky cantilever beam}

We consider the initial configuration in figure \ref{fig:initialNO}.
The volume of the structure under analysis is $V_0 = 45$ and we set the 
initial value of the Lagrange multiplier to $\gamma_0 = 0.1$.
In figure \ref{fig:NOholes}, we present the shapes obtained using the 
Boundary Variation Algorithm based on the expressions 
(\ref{eq:elas_shapeGrad_H1}) and (\ref{eq:elas_shapeGrad_mixed2}) of the 
shape gradient of the compliance. 
In particular, we remark that the variant of the BVA which exploits the 
shape gradient computed via the dual mixed variational formulation of the 
linear elasticity problem is able to construct configurations in which 
the total elastic energy is lower than in the corresponding cases 
obtained starting from the pure displacement formulation of the problem. 
This remark is confirmed by the comparison plots in figure 
\ref{fig:NOholes-comparison} where the BVA based on the dual mixed formulation 
is depicted by blue curves whereas the red ones represent the 
results obtained starting from the pure displacement formulation.
As a matter of fact, the former approach appears more robust than the 
latter one: the BVA based on the dual mixed formulation improves both the 
compliance and the functional $L(\Omega)$ during several iterations, 
whereas at the beginning of the evolution, 
the variant exploiting the pure displacement formulation reduces 
the compliance by enlarging the volume of the structure, thus 
deteriorating the corresponding value of $L(\Omega)$ (Fig. \ref{fig:NO-penal}). 
In a second phase, the BVA based on the pure displacement formulation is able 
to better control the variation of the volume and the final shapes obtained 
by the two algorithms have comparable sizes (Fig. \ref{fig:NO-vol}). 
Nevertheless, the overall improvement of the compliance is far more 
limited when using the pure displacement formulation with respect to the one 
observed starting from the dual mixed formulation (Fig. \ref{fig:NO-compli}).

\begin{figure}[htbp]
\begin{minipage}{0.88\columnwidth}
    \centering
    \subfloat[10 iterations.]
    {
    \includegraphics[width=0.3 \columnwidth]{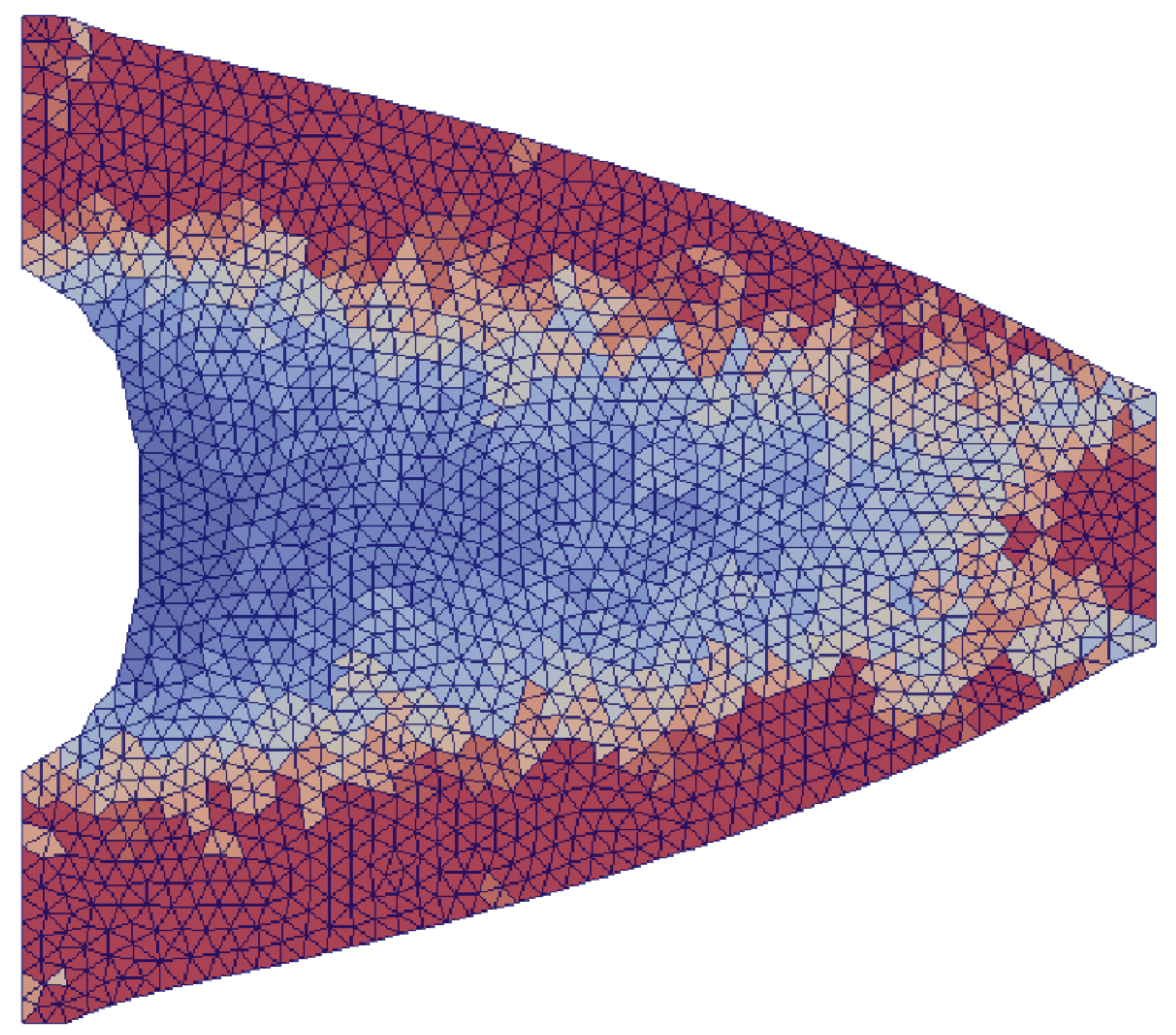}
    \label{fig:H1-NO-10}
    }
    \hfil
    \subfloat[20 iterations.]
    {
    \includegraphics[width=0.3 \columnwidth]{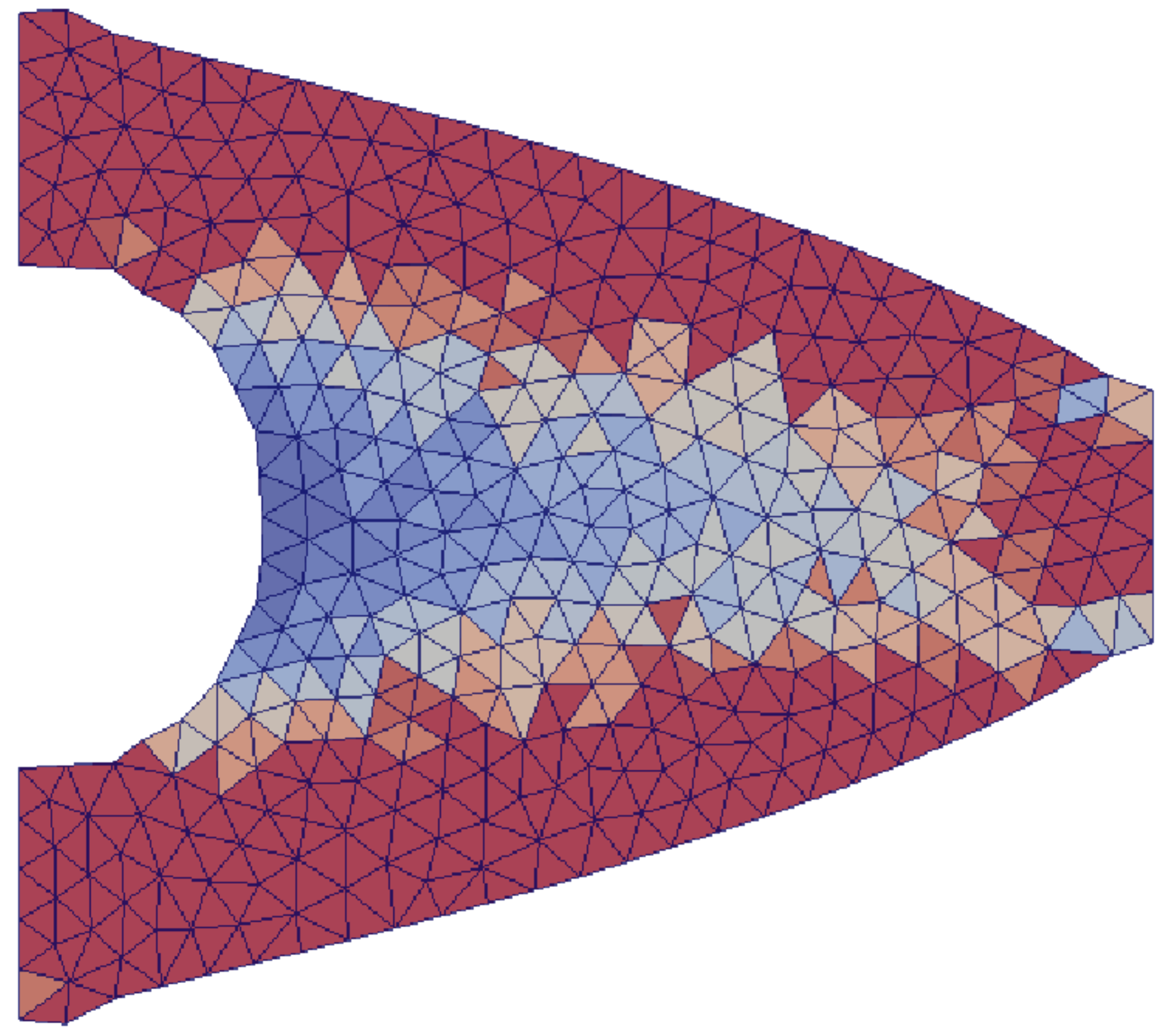}
    \label{fig:H1-NO-20}
    }
    \hfil
    \subfloat[30 iterations.]
    {
    \includegraphics[width=0.3 \columnwidth]{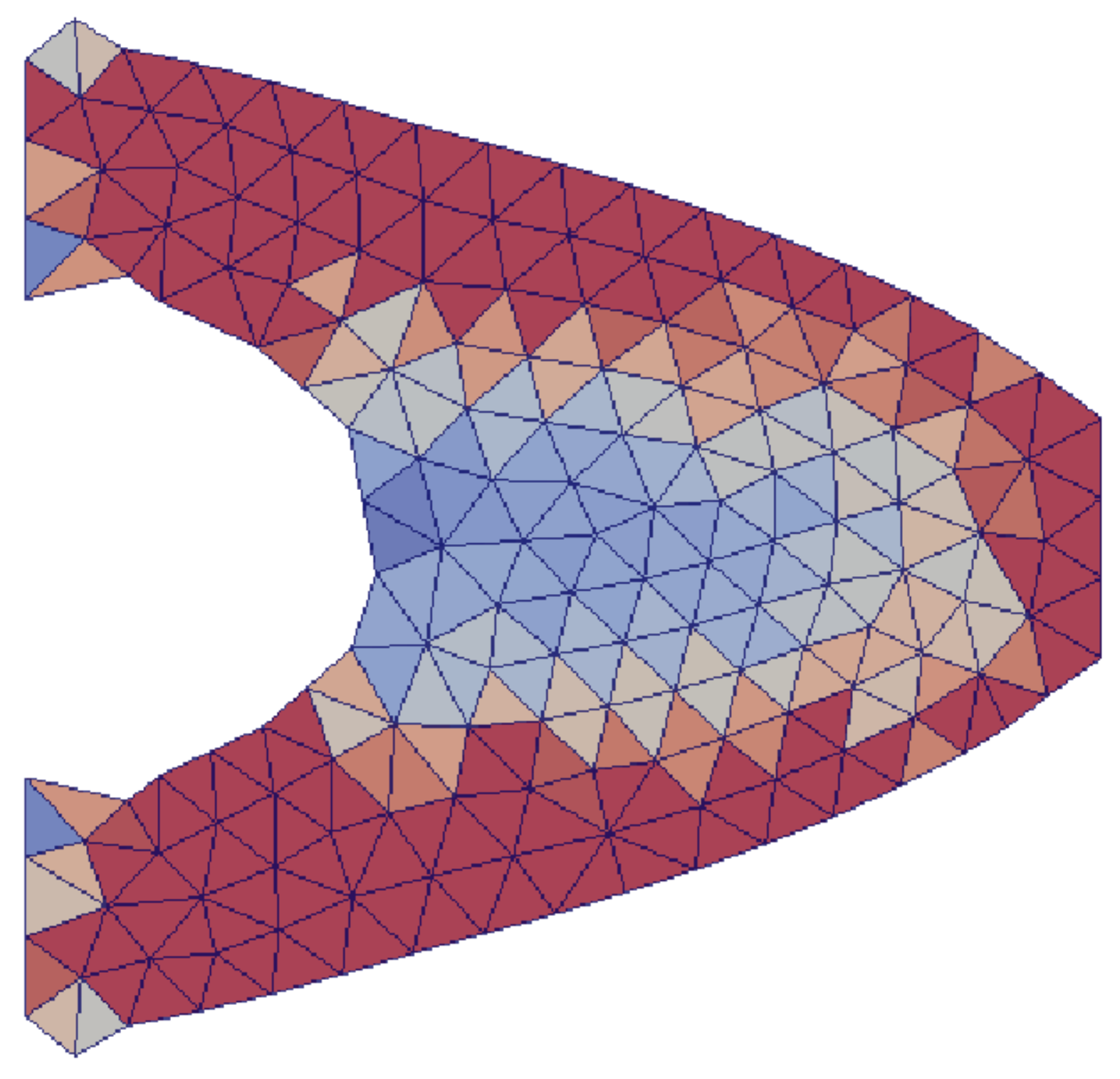}
    \label{fig:H1-NO-30}
    }
    
    \subfloat[10 iterations.]
    {
    \includegraphics[width=0.3 \columnwidth]{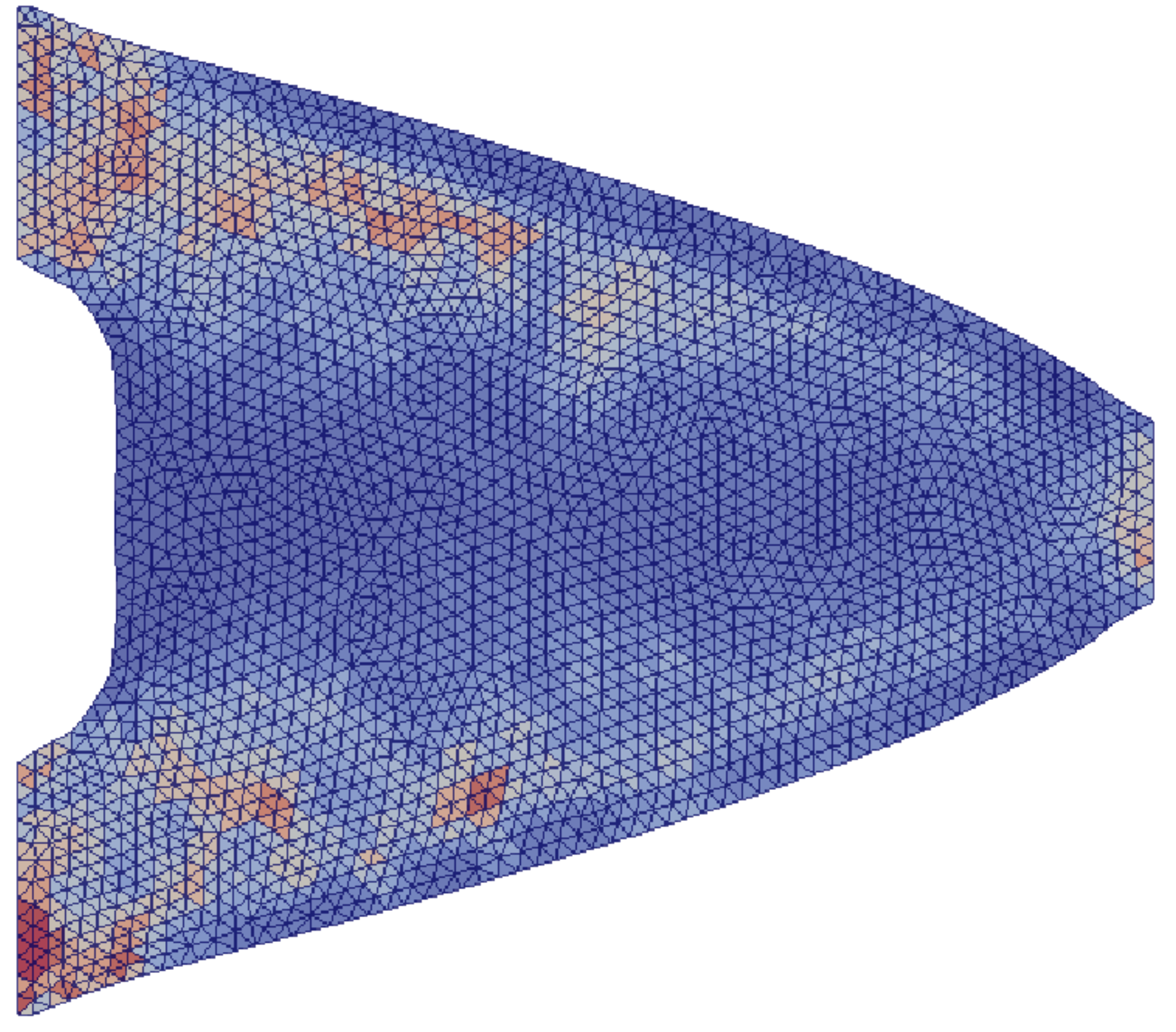}
    \label{fig:mix-NO-10}
    }
    \hfil
    \subfloat[20 iterations.]
    {
    \includegraphics[width=0.3 \columnwidth]{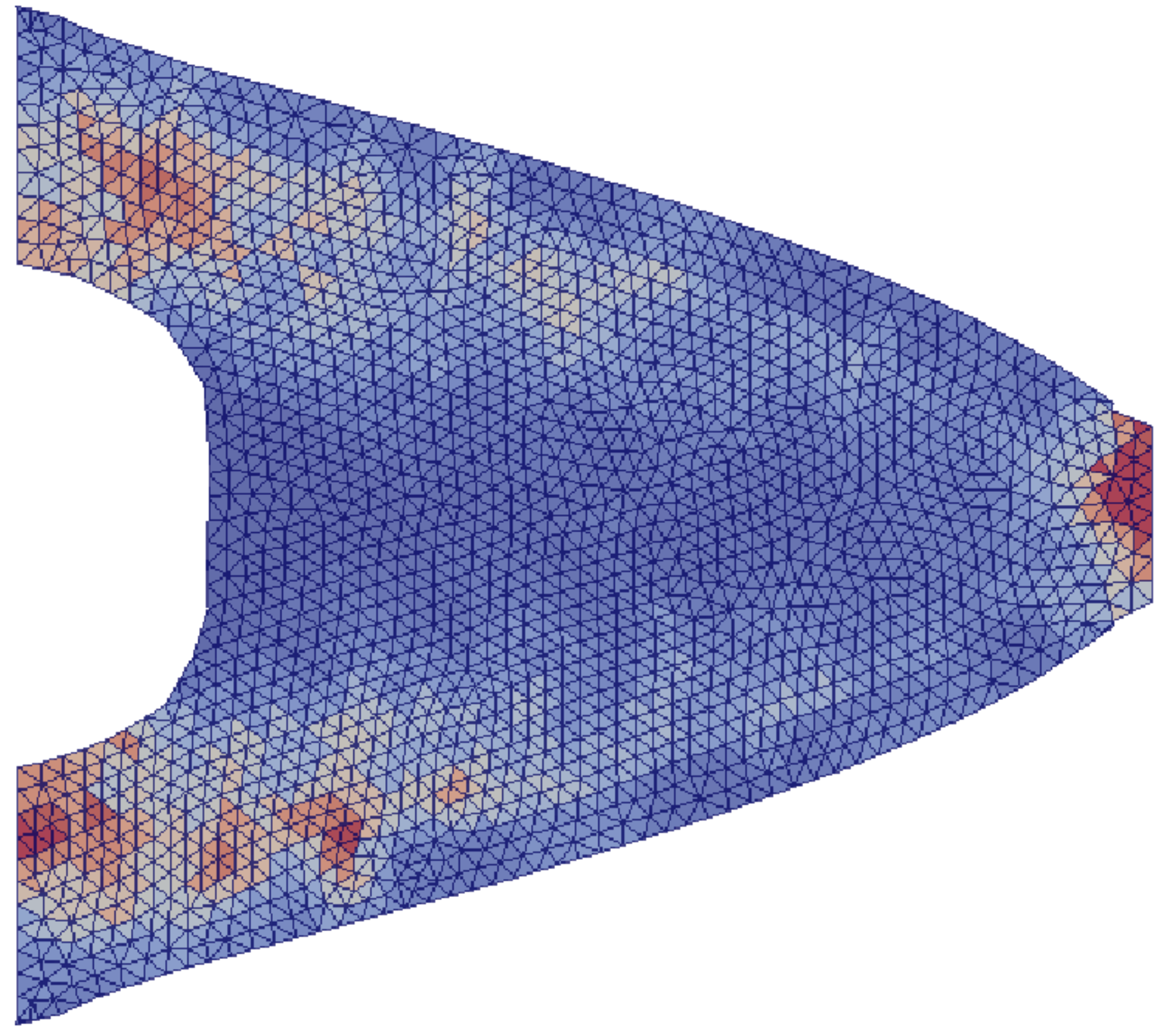}
    \label{fig:mix-NO-20}
    }
    \hfil
    \subfloat[30 iterations.]
    {
    \includegraphics[width=0.3 \columnwidth]{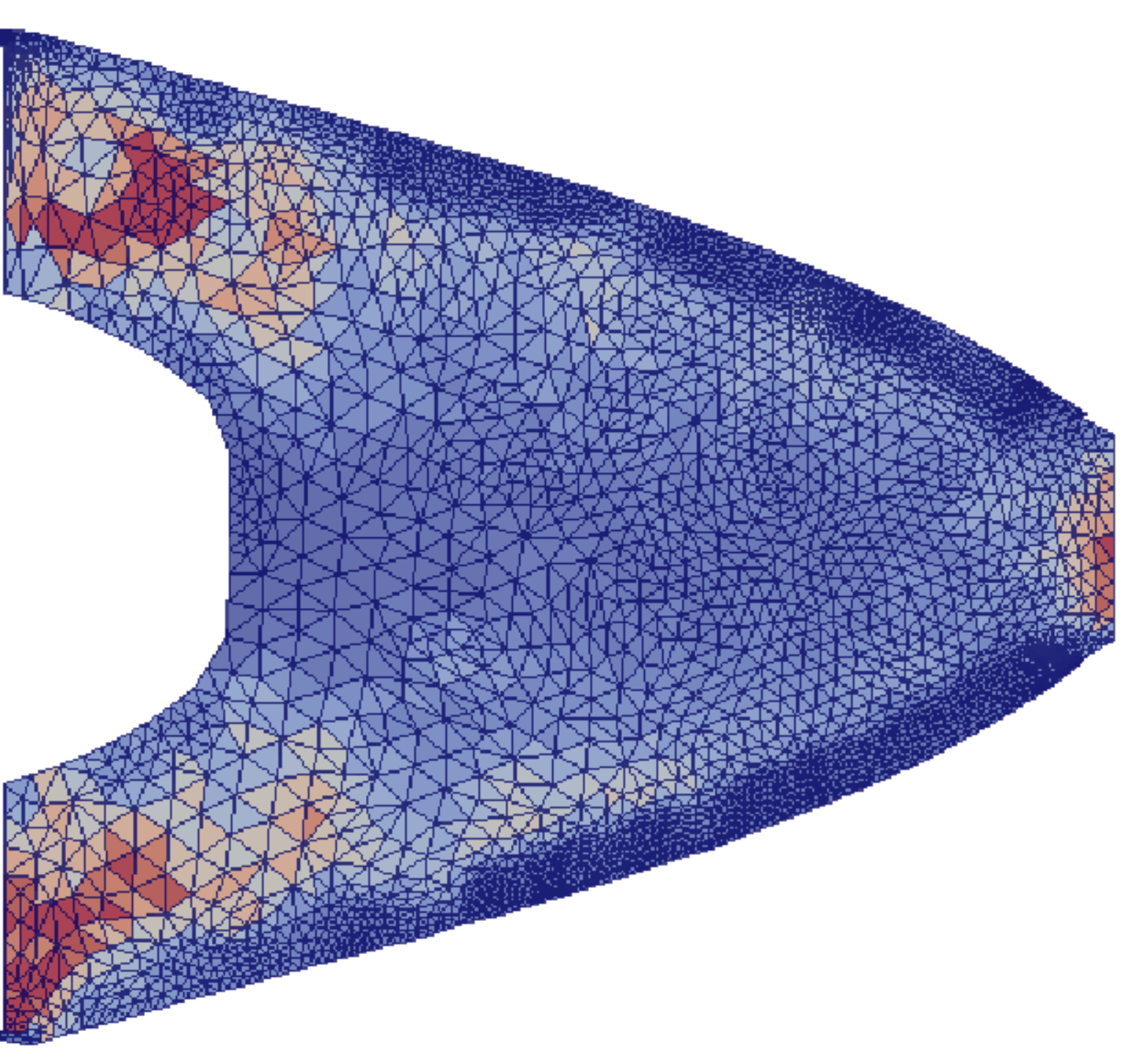}
    \label{fig:mix-NO-30}
    }
\end{minipage}
\begin{minipage}{0.1\columnwidth}
  \includegraphics[width=\columnwidth]{cantiNO-scale}
\end{minipage}
    \caption{Comparison of the BVA after 10, 20 and 30 iterations.
    At the top: BVA based on the expression of the shape gradient computed using 
    the pure displacement formulation of the linear elasticity problem.
    At the bottom: BVA using the shape gradient arising from the dual mixed variational formulation.
    Density distribution of the elastic energy within the range $(0,1.5 \cdot 10^{-3})$, 
    the lower values being in blue and the higher ones in red.}    
    \label{fig:NOholes}
\end{figure} 
\begin{figure}[hbtp]
    \centering
    \subfloat[Compliance $J(\Omega)$.]
    {
    \includegraphics[width=0.3 \columnwidth]{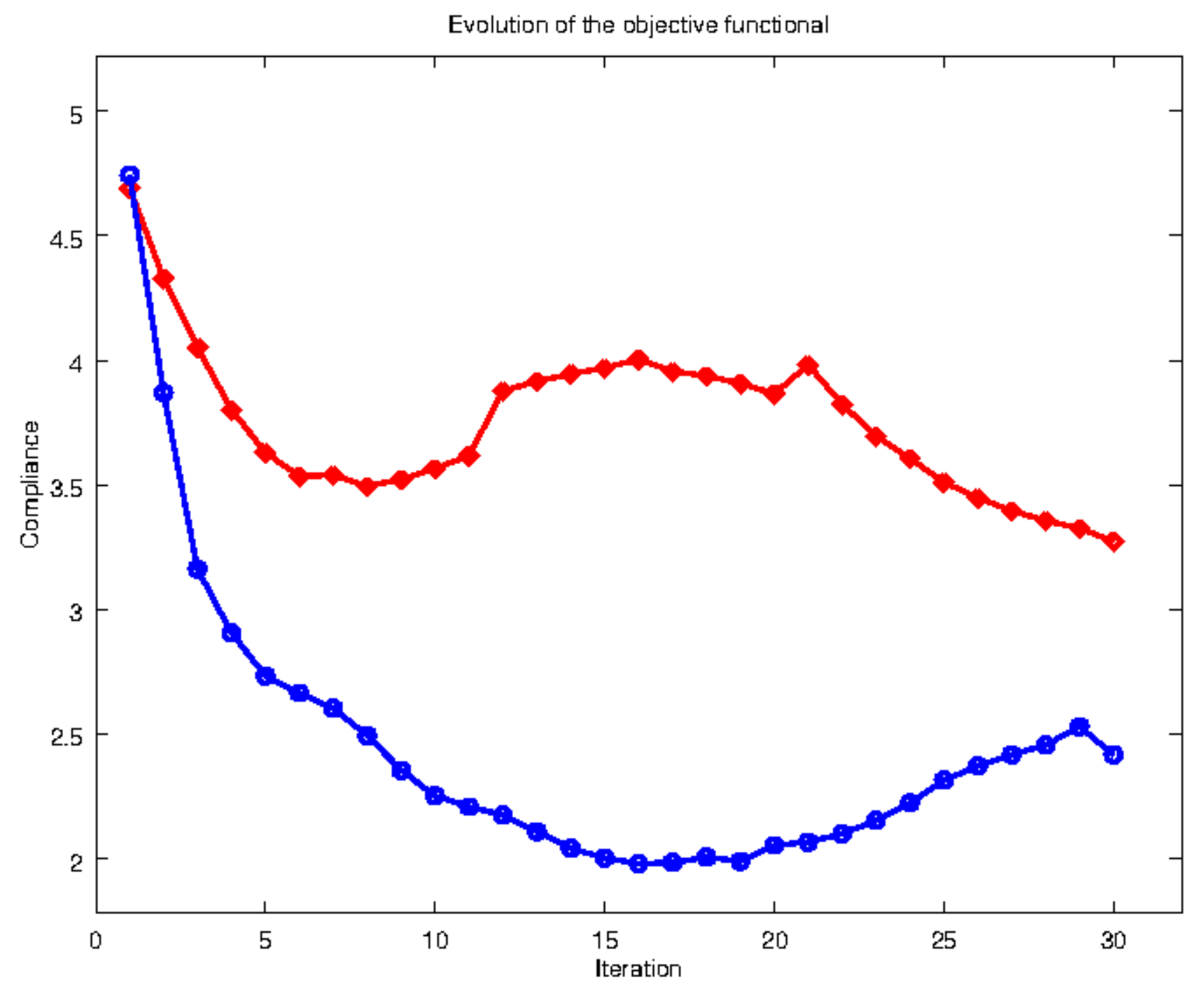}
    \label{fig:NO-compli}
    }
    \hfil
    \subfloat[Penalized functional $L(\Omega)$.]
    {
    \includegraphics[width=0.3 \columnwidth]{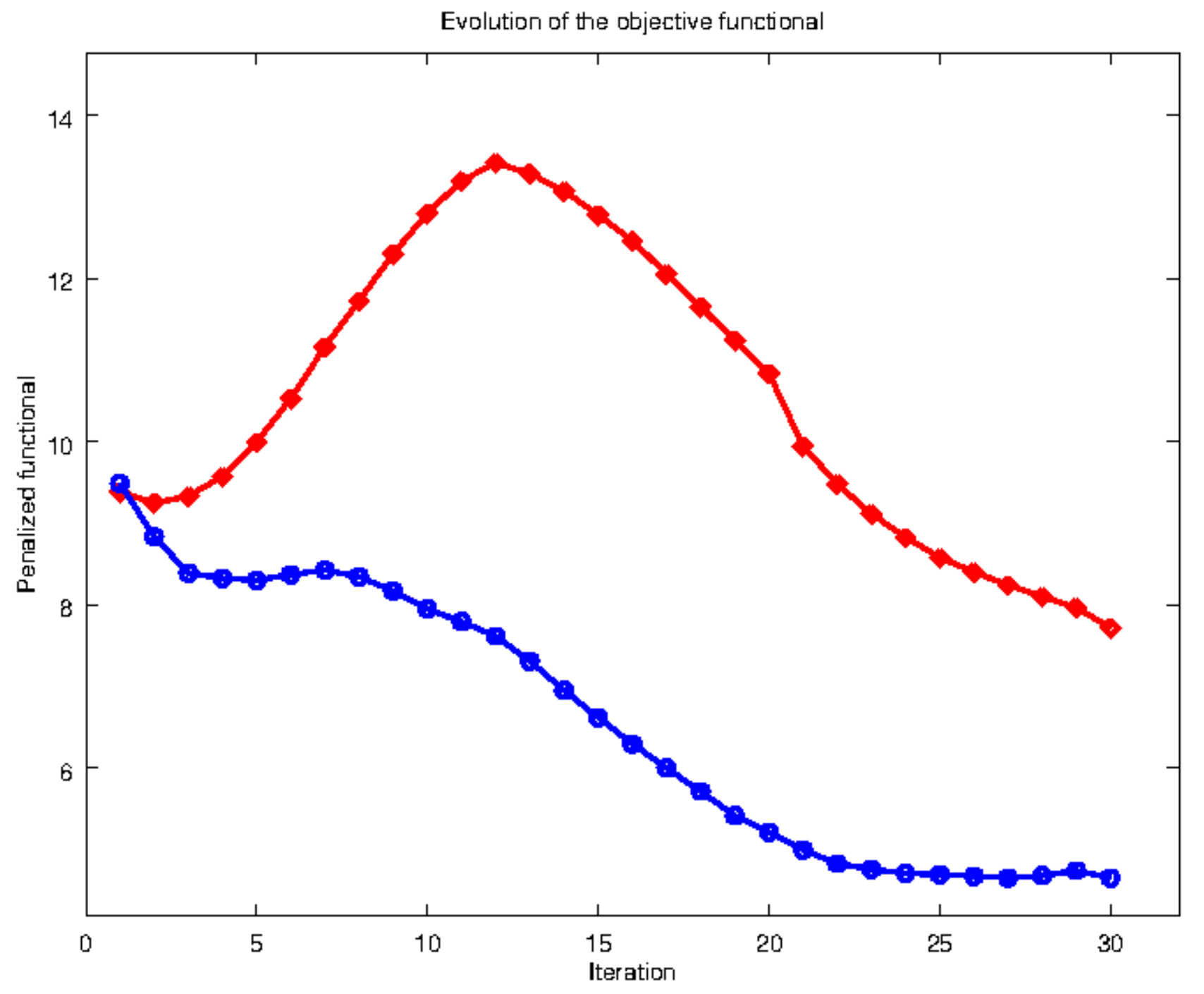}
    \label{fig:NO-penal}
    }
    \hfil
    \subfloat[Volume $V(\Omega)$.]
    {
    \includegraphics[width=0.3 \columnwidth]{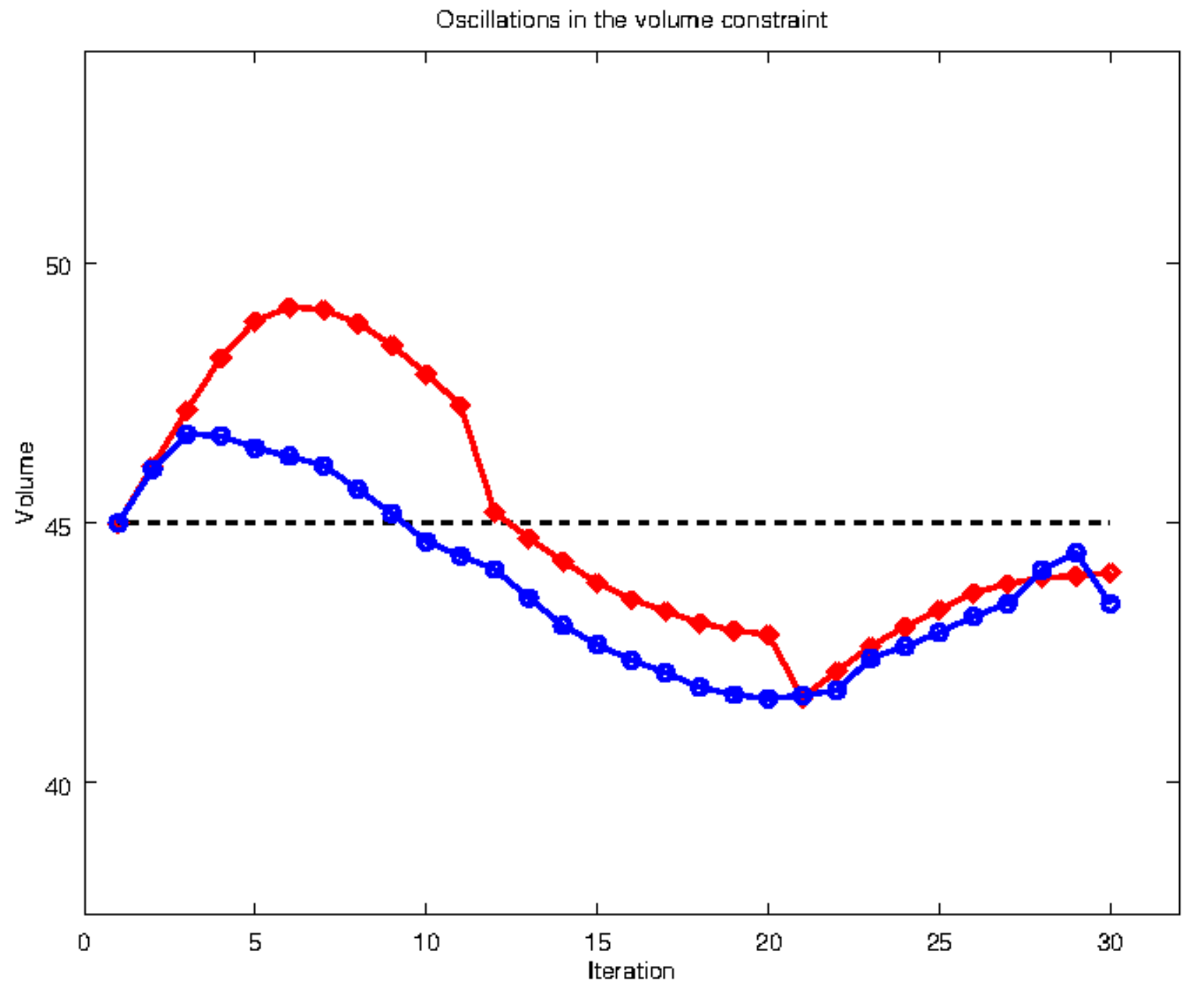}
    \label{fig:NO-vol}
    }
    \caption{Evolution of the (a) compliance $J(\Omega)$, (b) penalized functional 
    $L(\Omega)=J(\Omega)+\gamma V(\Omega)$ and (c) volume $V(\Omega)$ using the 
    BVA. Results obtained using the pure displacement 
    formulation (red diamond) and the dual mixed one (blue circle). 
    The reference volume $V_0$ is represented by a black dashed line in (c).}    
    \label{fig:NOholes-comparison}
\end{figure}

\subsubsection*{Cantilever beam with six holes}

The initial shape for the cantilever beam with six holes 
is depicted in figure \ref{fig:initial6} and features a 
reference volume $V_0 = 40.59$ and an initial Lagrange 
multiplier equal to $\gamma_0 = 0.13$.
As for the case of the bulky cantilever, we present 
snapshots of the shapes obtained at different iterations 
of the Boundary Variation Algorithm using both the pure 
displacement and the dual mixed formulation of the linear 
elasticity problem (Fig. \ref{fig:6holes}). 
Moreover, a qualitative analysis of the evolution of the 
compliance and of the variation of the volume is 
discussed starting from figure \ref{fig:6holes-comparison}.
As previously remarked, the Boundary Variation Algorithm 
based on the dual mixed formulation of the linear elasticity 
problem leads to configurations with lower elastic energy.
Figures \ref{fig:6-compli} and \ref{fig:6-penal} confirm that 
the variant of the BVA using the dual mixed formulation 
generates a sequence of shapes that improve the objective 
functional for several subsequent iterations.
On the contrary, the pure displacement formulation leads to a 
less robust strategy in which at the beginning of the 
optimization process, the compliance is reduced by 
increasing the volume of the structure.
Concerning the BVA based on the dual mixed formulation, 
the comparison of figure \ref{fig:mix-6-20} with 
figure \ref{fig:mix-6-30}, highlights that 
only minor modifications of the shape are performed 
by the algorithm from iteration 20 to iteration 30. 
As a matter of fact, the evolution of the volume (Fig. \ref{fig:6-vol}) 
shows that after having identified a configuration with 
low compliance the algorithm tends to correct the shape 
in order to fulfill the volume constraint which has been 
violated during the initial iterations.
As highlighted by the test case of the bulky cantilever, 
the Boundary Variation Algorithm based on the dual mixed formulation 
is able to construct structures with lower compliance than the 
configurations generated using the pure displacement formulation 
(Fig. \ref{fig:6-compli}).
Nevertheless, both the final configuration in figure \ref{fig:H1-6-30} 
and the one in figure \ref{fig:mix-6-30}, present some 
issues. On the one hand, the pure displacement solution presents 
kinks responsible for low compliance near the regions $\Gamma^D$ 
where the structure is clamped.
On the other hand, the shape obtained by the dual mixed formulation 
features thin components which may be critical to handle 
during the manufacturing process. 
Both these issues may be potentially influenced by the choice of 
explicitly representing the geometry through the computational mesh 
and the consequent moving mesh approach to deform the domain.
In order to bypass these issues, an implicit description of the geometry 
may be employed.

\begin{figure}[hbtp]
\begin{minipage}{0.88\columnwidth}
    \centering
    \subfloat[10 iterations.]
    {
    \includegraphics[width=0.3 \columnwidth]{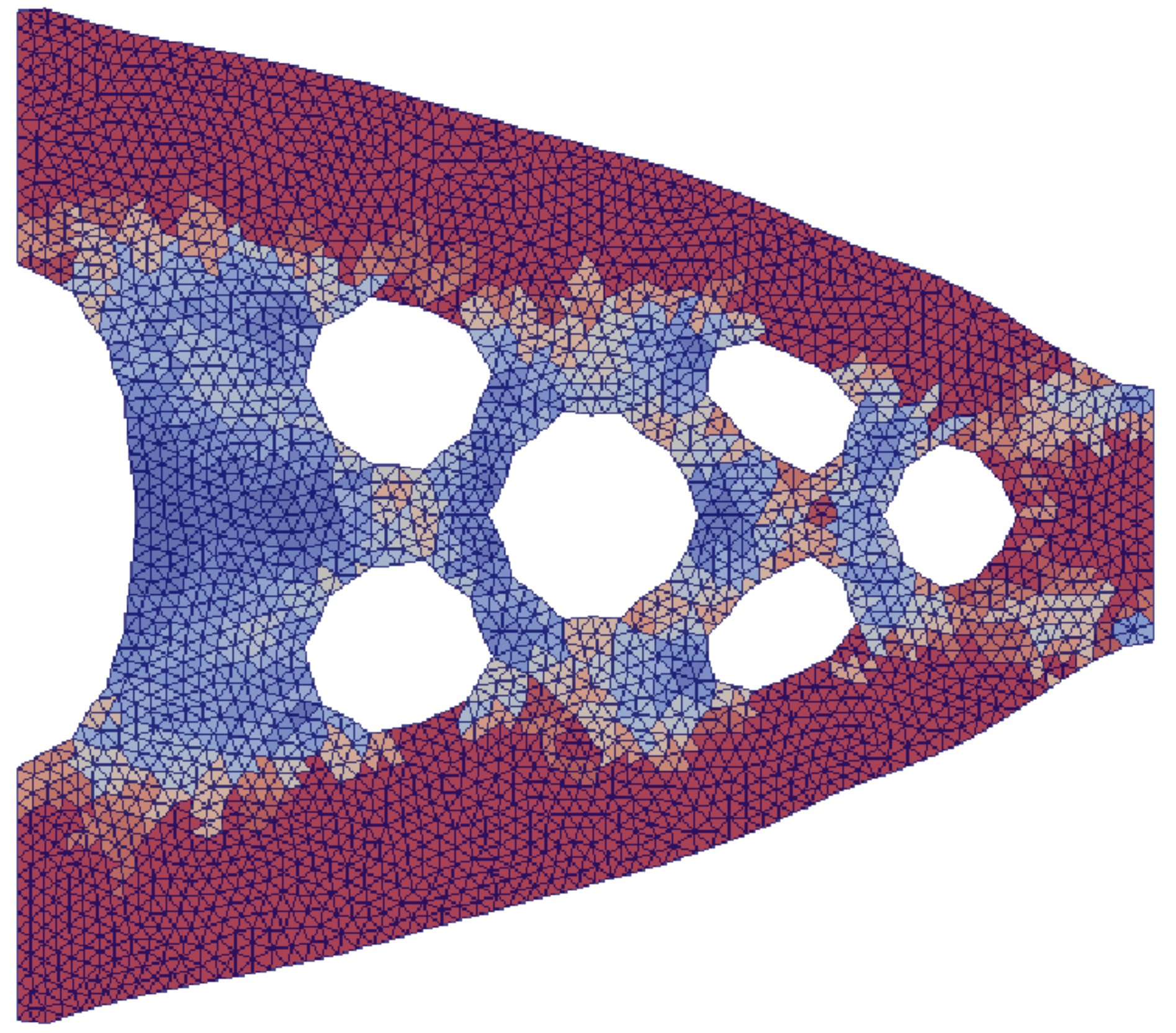}
    \label{fig:H1-6-10}
    }
    \hfil
    \subfloat[20 iterations.]
    {
    \includegraphics[width=0.3 \columnwidth]{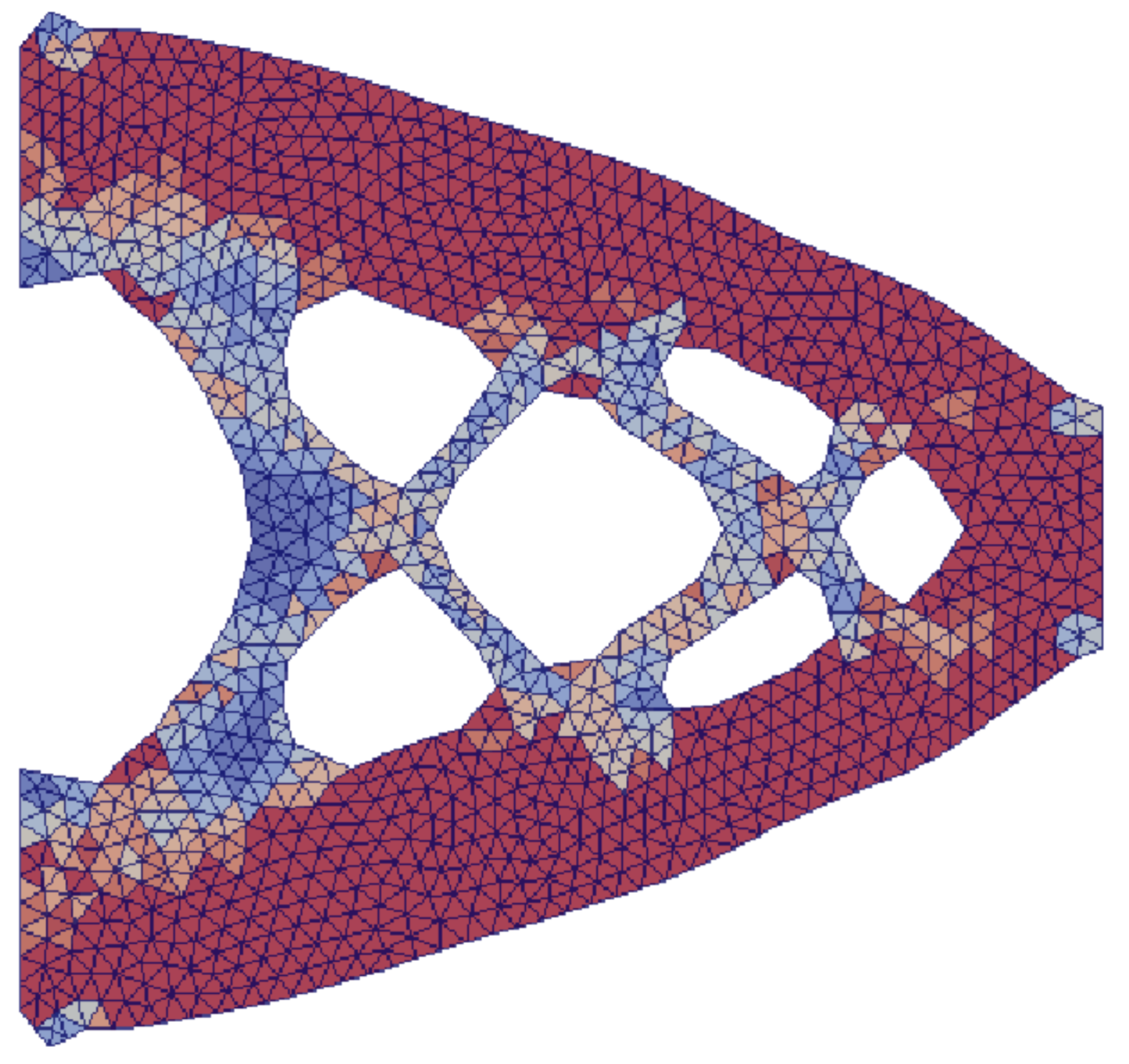}
    \label{fig:H1-6-20}
    }
    \hfil
    \subfloat[30 iterations.]
    {
    \includegraphics[width=0.3 \columnwidth]{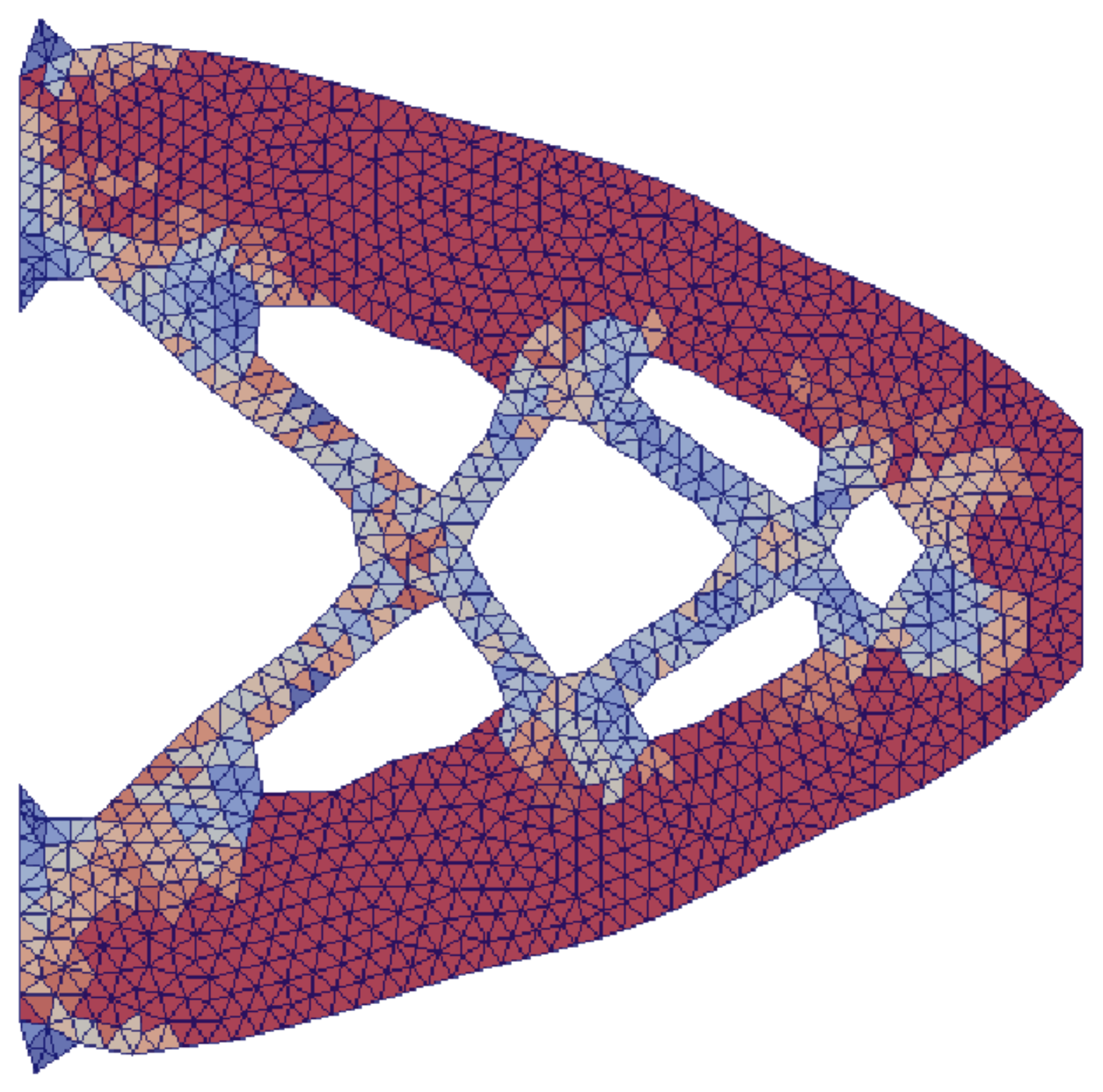}
    \label{fig:H1-6-30}
    }
    
    \subfloat[10 iterations.]
    {
    \includegraphics[width=0.3 \columnwidth]{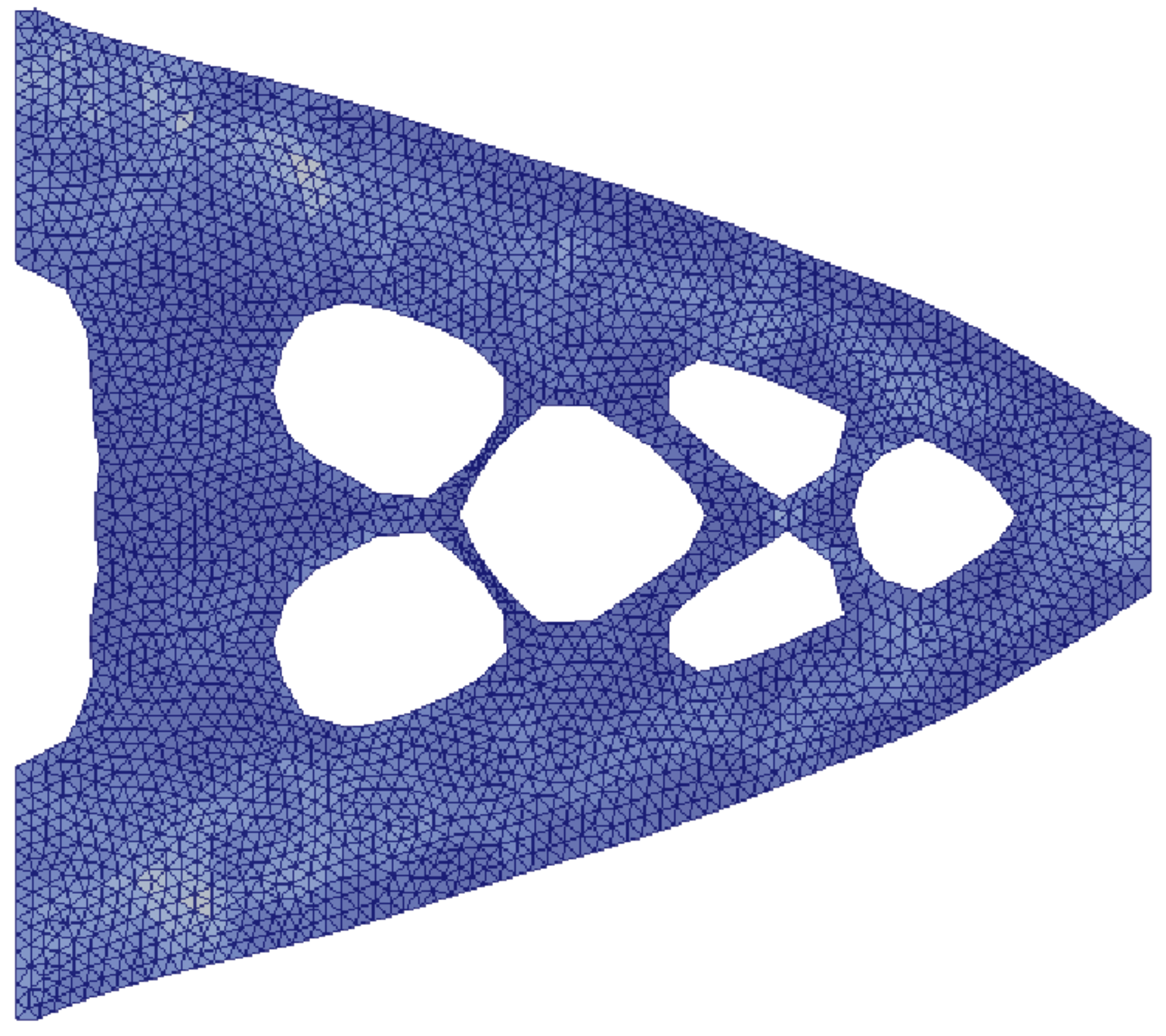}
    \label{fig:mix-6-10}
    }
    \hfil
    \subfloat[20 iterations.]
    {
    \includegraphics[width=0.3 \columnwidth]{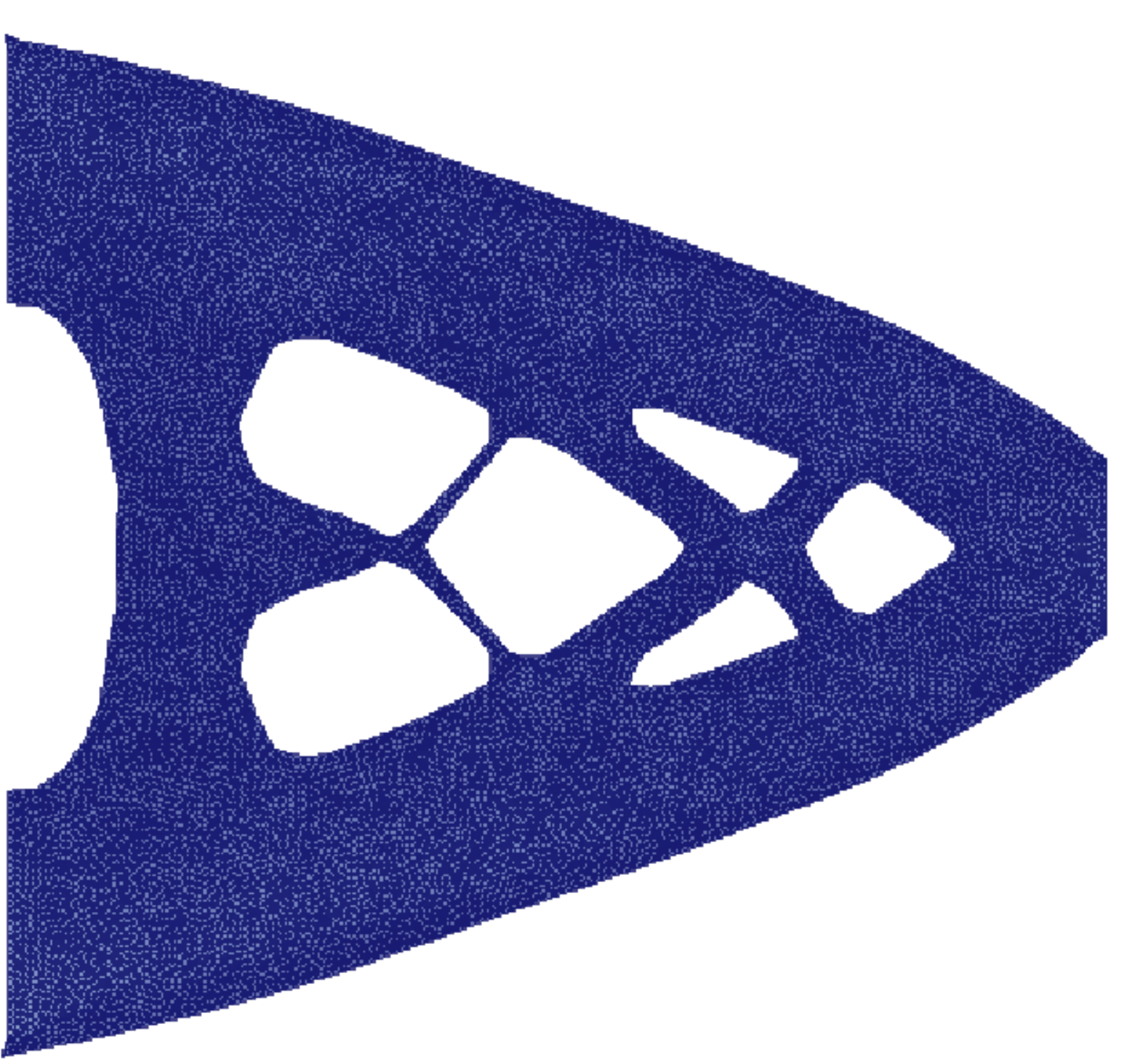}
    \label{fig:mix-6-20}
    }
    \hfil
    \subfloat[30 iterations.]
    {
    \includegraphics[width=0.3 \columnwidth]{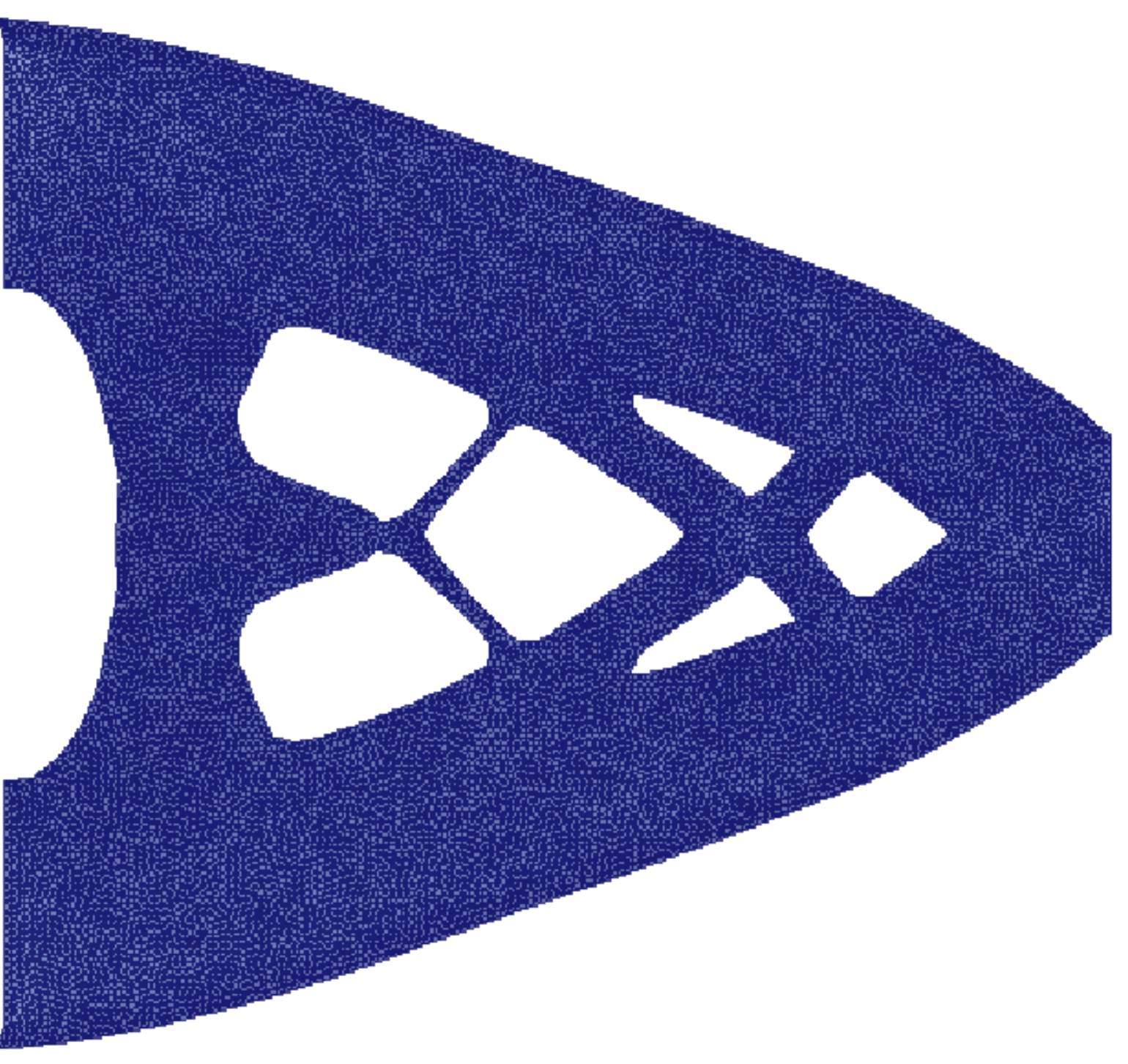}
    \label{fig:mix-6-30}
    }
\end{minipage}
\begin{minipage}{0.1\columnwidth}
  \includegraphics[width=\columnwidth]{canti6-scale}
\end{minipage}
    \caption{Comparison of the BVA after 10, 20 and 30 iterations.
    At the top: BVA based on the expression of the shape gradient computed using 
    the pure displacement formulation of the linear elasticity problem.
    At the bottom: BVA using the shape gradient arising from the dual mixed variational formulation.
    Density distribution of the elastic energy within the range $(0,3 \cdot 10^{-3})$, 
    the lower values being in blue and the higher ones in red.}    
    \label{fig:6holes}
\end{figure} 
\begin{figure}[hbtp]
    \centering
    \subfloat[Compliance $J(\Omega)$.]
    {
    \includegraphics[width=0.3 \columnwidth]{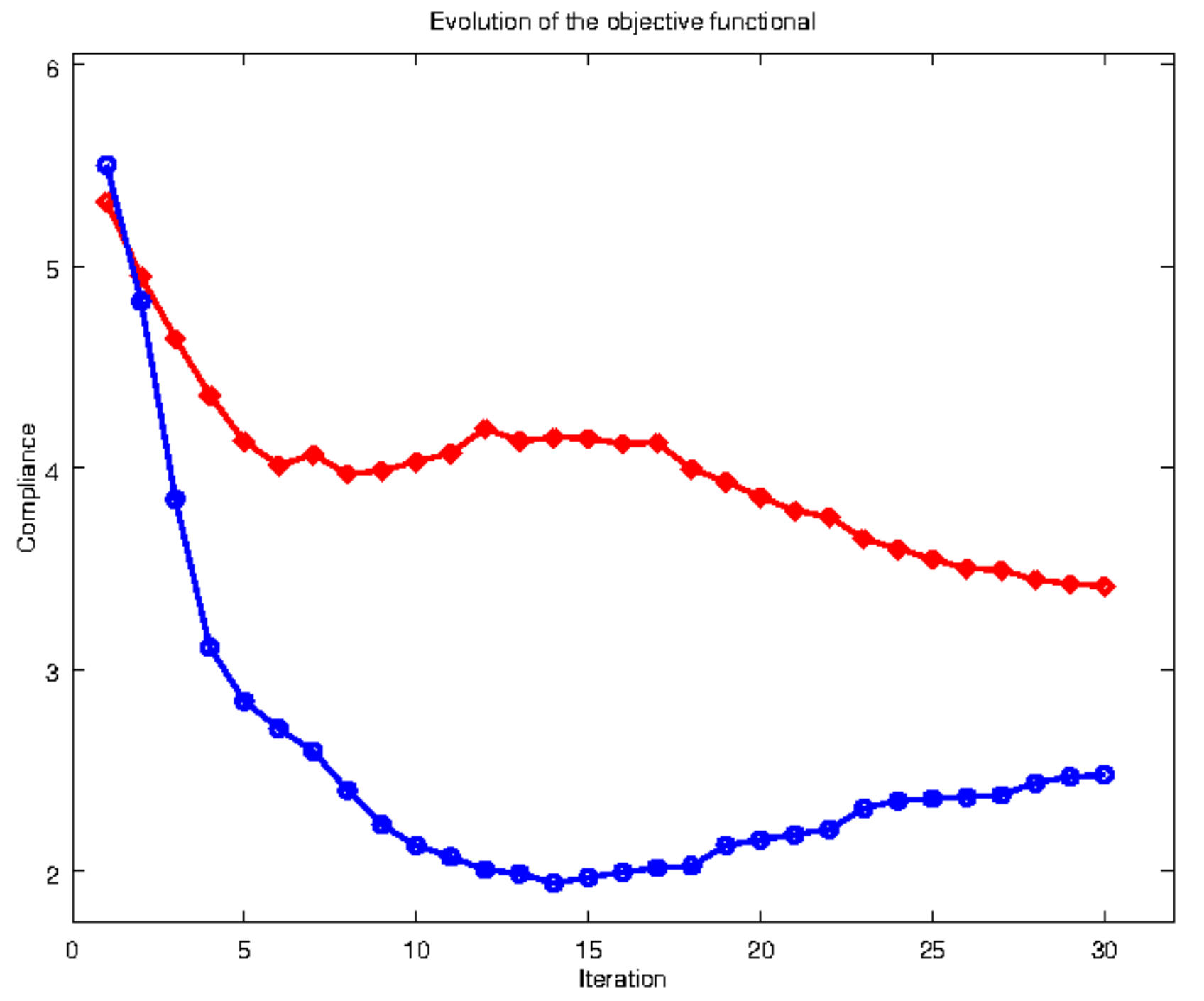}
    \label{fig:6-compli}
    }
    \hfil
    \subfloat[Penalized functional $L(\Omega)$.]
    {
    \includegraphics[width=0.3 \columnwidth]{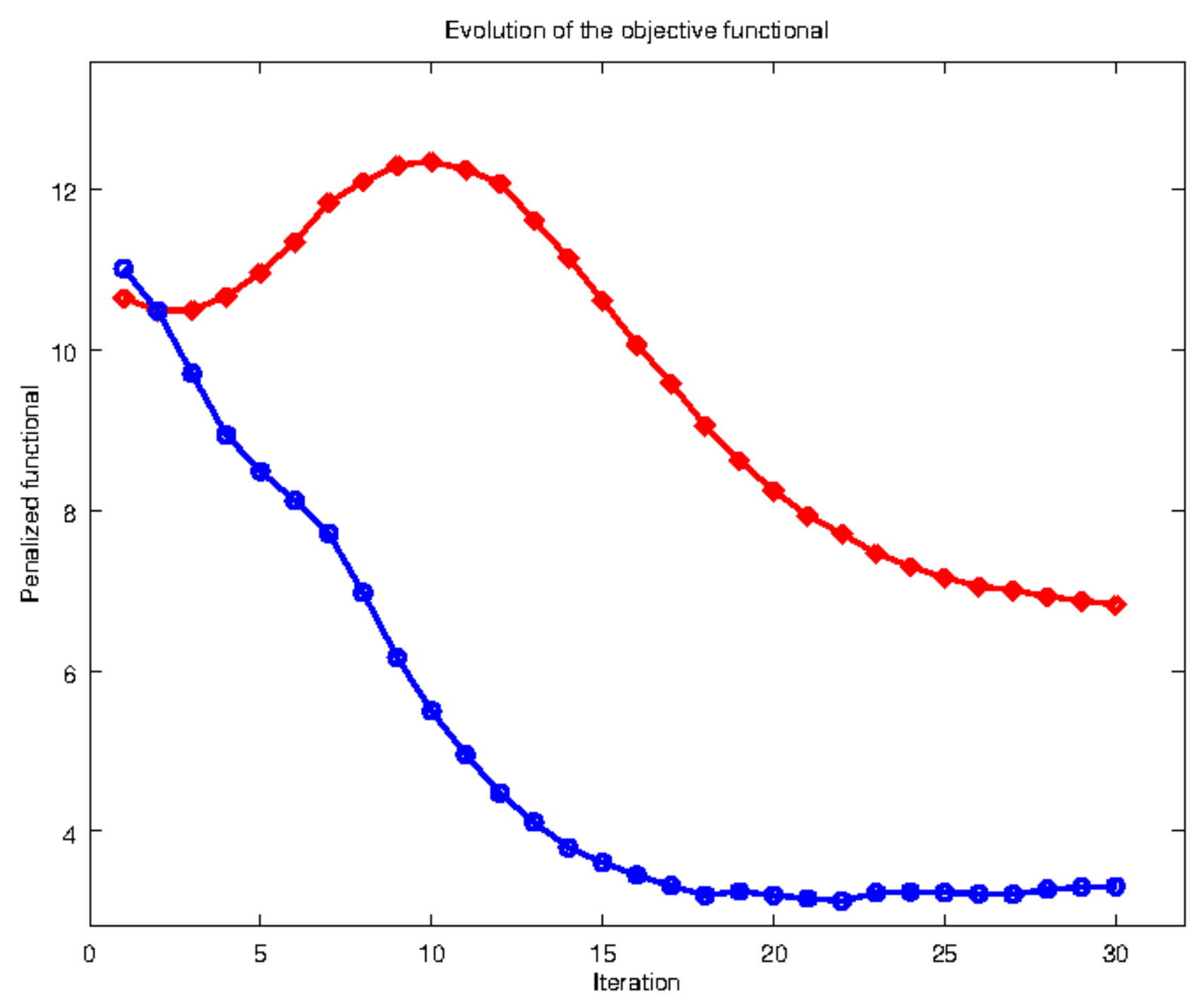}
    \label{fig:6-penal}
    }
    \hfil
    \subfloat[Volume $V(\Omega)$.]
    {
    \includegraphics[width=0.3 \columnwidth]{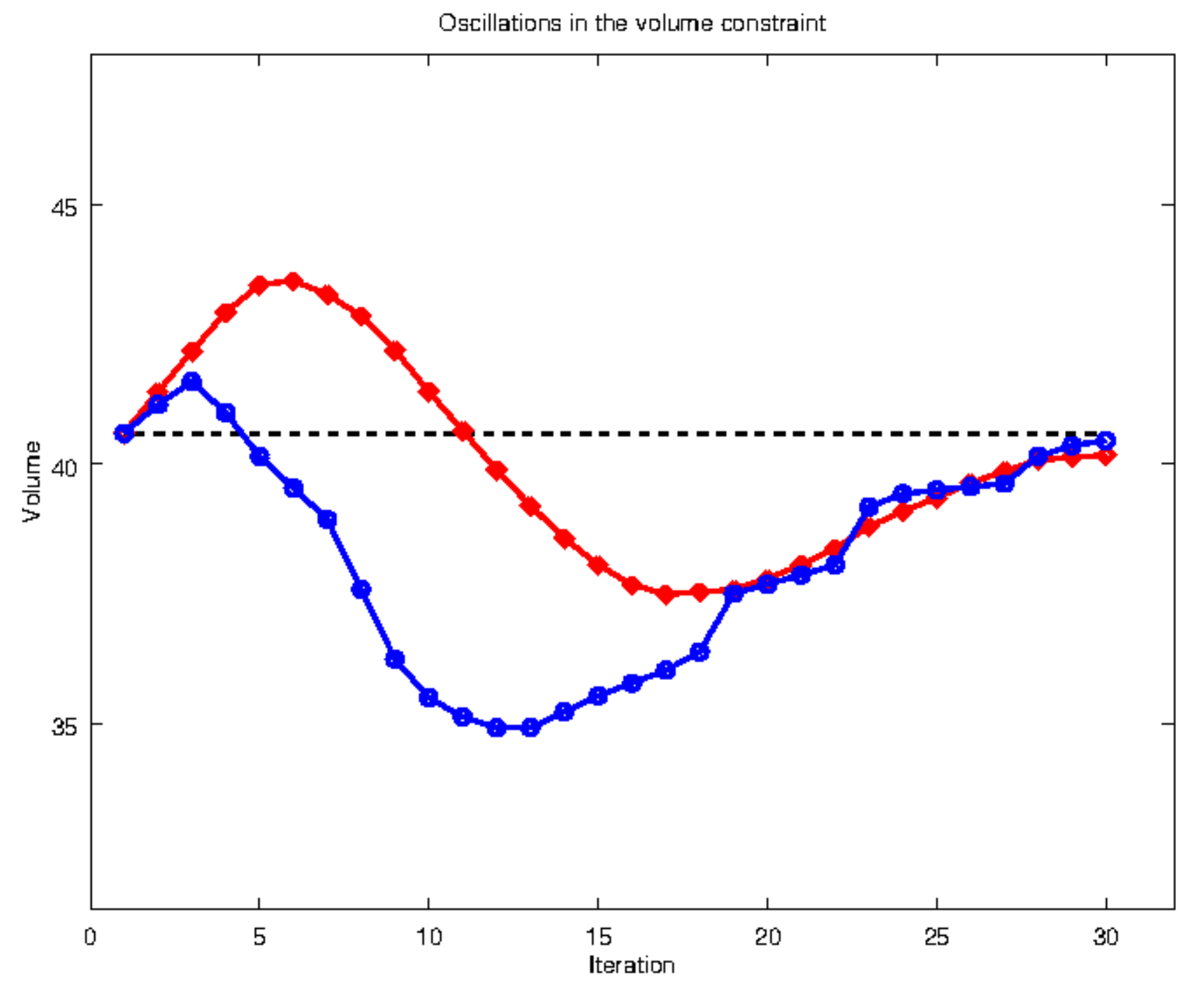}
    \label{fig:6-vol}
    }
    \caption{Evolution of the (a) compliance $J(\Omega)$, (b) penalized functional 
    $L(\Omega)=J(\Omega)+\gamma V(\Omega)$ and (c) volume $V(\Omega)$ using the 
    BVA. Results obtained using the pure displacement 
    formulation (red diamond) and the dual mixed one (blue circle). 
    The reference volume $V_0$ is represented by a black dashed line in (c).}    
    \label{fig:6holes-comparison}
\end{figure}

Concerning the computational cost of the overall optimization procedures, 
it is important to remark that the dual mixed formulation features more 
variables (stress tensor $\sigma_\Omega$, displacement field $u_\Omega$ 
and Lagrange multiplier $\eta_\Omega$) than the pure displacement one 
which - as the name states - solely relies on the displacement field $u_\Omega$.
From a practical point of view, this results in a higher number of Degrees 
of Freedom in the discrete problem and consequently a higher computational 
cost. 
Moreover, by comparing the first and the second lines of figures \ref{fig:NOholes} 
and \ref{fig:6holes}, 
we remark that the computations of the BVA based on the dual mixed formulation 
were performed on finer meshes than the ones used for the pure 
displacement one. 
This turned out to be necessary in order to retrieve an accurate solution of the 
dual mixed Finite Element problem of linear elasticity, whereas the  
pure displacement formulation may be easily approximated using Lagrangian Finite 
Element functions as long as one avoids the nearly incompressible and the 
incompressible case.
Eventually, the linear system obtained by the discretization discussed in 
subsection \ref{ref:mixed_weak} may be extremely ill-posed and the construction of 
appropriate preconditioners (cf. e.g. \cite{NLA:NLA357}) may be necessary.
Hence, though the preliminary numerical results suggest that the BVA based on 
the dual mixed formulation is the best choice when dealing with the minimization 
of the compliance in linear elasticity, the higher computational cost 
and the additional numerical difficulties of the overall strategy 
have to be taken into account to provide a global evaluation 
of the method.
Within this context, additional investigations have to be performed 
both from a theoretical point of view (e.g. \emph{a priori} estimate of the 
error in the shape gradient) and from a computational one, by optimizing 
and improving the resolution strategy outlined above.

\section{Conclusion}
\label{ref:conclusion}

To the best of our knowledge, the results in this article are 
the first attempt to derive volumetric expressions of the shape 
gradient of a shape-dependent functional within the framework of linear elasticity.
In particular, we computed two novel expressions of the shape gradient 
of the compliance starting from the pure displacement and the dual mixed 
formulations of the governing equation.
A preliminary comparison of the aforementioned expressions 
by means of numerical simulations showed extremely promising results, 
especially using the dual mixed variational formulation of the 
linear elasticity equation.
As a matter of fact, the global optimization strategy based on this approach 
seems more robust than the one obtained from the pure displacement formulation 
and is able to further reduce the compliance of the structure under analysis.
Nevertheless, a rigorous and detailed analysis both from an 
analytical and a numerical point of view is necessary to validate 
the aforementioned statement.
Concerning the analytical derivation of the volumetric shape gradient of the compliance, 
a rigorous proof of the equivalence of the expressions obtained using the 
pure displacement and the dual mixed variational formulations of the linear elasticity 
problem is required.
Moreover, following the analysis performed for the elliptic case in \cite{HPS_bit}, 
\emph{a priori} estimates of the error in the shape gradient may be derived. 
This analysis seems particularly interesting since it may provide additional 
information on the convergence of the shape gradient using different discretization 
techniques, thus possibly fostering one formulation over the other to achieve 
better accuracy in the approximation of the shape gradient.

\begin{table}[htb]
    \footnotesize
    \centering
    \subfloat[Test in fig. \ref{fig:NOholes-comparison} - Pure displacement formulation.]
    {
    \centering
    \begin{tabular}[hbt]{| c | c || c | c |}
    \hline
    It. $j$ & $L(\Omega_j)$ & $\langle d_h L(\Omega_j),\theta_j^h \rangle$ & $L(\Omega_{j+1})$ \\
    \hline & & & 
    \\ [-1em] \hline
    1 & $9.38$ & $-5.7 \cdot 10^{-1}$ & $9.25$ \\
    \rowcolor{yellow} \hline
    5 & $9.99$ & $-6.61 \cdot 10^{-2}$ & $10.52$ \\
    \rowcolor{yellow} \hline
    10 & $12.80$ & $-1.05 \cdot 10^{-1}$  & $13.19$ \\
    \hline
    15 & $12.78$ & $-9.4 \cdot 10^{-2}$ & $12.46$ \\
    \hline
    25 & $8.58$ & $-3.61 \cdot 10^{-2}$ & $8.4$ \\
    \hline    
    30 & $7.54$ & $-9.52 \cdot 10^{-3}$ & $-$ \\
    \hline
    \end{tabular}
    \label{tab:NOholesH1}
    }
    \subfloat[Test in fig. \ref{fig:NOholes-comparison} - Dual mixed formulation.]
    {
    \centering
    \begin{tabular}[hbt]{| c | c || c | c |}
    \hline
    It. $j$ & $L(\Omega_j)$ & $\langle d_h L(\Omega_j),\theta_j^h \rangle$ & $L(\Omega_{j+1})$ \\
    \hline & & & 
    \\ [-1em] \hline
    1 & $9.49$ & $-4.78$ & $8.83$ \\
    \rowcolor{yellow} \hline 
    6 & $8.37$ & $-7.28 \cdot 10^{-1}$ & $8.43$ \\
    \hline
    10 & $7.95$ & $-7.88 \cdot 10^{-1}$ & $7.79$ \\
    \hline
    15 & $6.62$ & $-5.55 \cdot 10^{-1}$ & $6.3$ \\
    \rowcolor{yellow} \hline
    28 & $4.68$ & $-6.55 \cdot 10^{-1}$ & $4.74$ \\
    \hline    
    30 & $4.66$ & $-6.21 \cdot 10^{-1}$ & $-$ \\
    \hline
    \end{tabular}
    \label{tab:NOholesMix}
    }

    \subfloat[Test in fig. \ref{fig:6holes-comparison} - Pure displacement formulation.]
    {
    \centering
    \begin{tabular}[htb]{| c | c || c | c |}
    \hline
    It. $j$ & $L(\Omega_j)$ & $\langle d_h L(\Omega_j),\theta_j^h \rangle$ & $L(\Omega_{j+1})$ \\
    \hline & & & 
    \\ [-1em] \hline
    1 & $10.64$ & $-1.00$ & $10.48$ \\
    \rowcolor{yellow} \hline
    3 & $10.50$ & $-3.68  \cdot 10^{-1}$ & $10.66$ \\
    \rowcolor{yellow} \hline
    9 & $12.30$ & $-1.74 \cdot 10^{-1}$  & $12.34$ \\
    \hline
    15 & $10.61$ & $-1.04 \cdot 10^{-1}$ & $10.06$ \\
    \hline
    25 & $7.16$ & $-1.1 \cdot 10^{-2}$ & $7.05$ \\
    \hline    
    30 & $6.82$ & $-2.75 \cdot 10^{-3}$ & $-$ \\
    \hline
    \end{tabular}
    \label{tab:6holesH1}
    }
    \subfloat[Test in fig. \ref{fig:6holes-comparison} - Dual mixed formulation.]
    {
    \centering
    \begin{tabular}[htb]{| c | c || c | c |}
    \hline
    It. $j$ & $L(\Omega_j)$ & $\langle d_h L(\Omega_j),\theta_j^h \rangle$ & $L(\Omega_{j+1})$ \\
    \hline & & & 
    \\ [-1em] \hline
    1 & $11.01$ & $-15.32$ & $10.48$ \\
    \hline
    5 & $8.49$ & $-2.55$ & $8.12$ \\
    \hline
    10 & $5.50$ & $-1.29$ & $4.95$ \\
    \rowcolor{yellow} \hline
    18 & $3.18$ & $-1.36$ & $3.19$ \\
    \rowcolor{yellow} \hline
    28 & $3.32$ & $-2.06$ & $3.36$ \\
    \hline    
    30 & $3.37$ & $-2.18$ & $-$ \\
    \hline
    \end{tabular}
    \label{tab:6holesMix}
    }
\caption{Boundary Variation Algorithm based on the pure displacement formulation 
(left) and on the dual mixed formulation (right) of the linear elasticity problem.
On the first line: test case in figure \ref{fig:NOholes-comparison}. 
On the second line: test case in figure \ref{fig:6holes-comparison}. 
Evolution of the penalized objective functional $L(\Omega)$ with respect to 
the iteration number. 
In yellow: the cases in which the discretized direction $\theta^h$ fails 
to be a genuine descent direction for $L(\Omega)$ despite being 
$\langle d_h L(\Omega), \theta^h \rangle < 0$.}
\label{tab:NO6holes}
\end{table}

The numerical results in section \ref{ref:shapeGrad_comparison} 
highlight some issues associated with the application of the Boundary 
Variation Algorithm to the minimization of the compliance in structural 
optimization. 
On the one hand, it is straightforward to observe 
(Fig. \ref{fig:NOholes-comparison} and \ref{fig:6holes-comparison}) 
that the direction computed using the 
discretized shape gradient is not always a genuine descent direction for the 
functional under analysis. To remedy this issue, in 
\cite{1742-6596-657-1-012004, giacomini:hal-01201914, giacomini:flux} we 
proposed a variant of the BVA - named Certified Descent Algorithm (CDA) - that 
couples a gradient-based optimization strategy with \emph{a posteriori} 
estimators of the error in the shape gradient.
This remark is confirmed by table \ref{tab:NO6holes} in which we observe that despite 
being $\langle d_h L(\Omega), \theta^h \rangle < 0$, 
the functional $L(\Omega)$ may increase when the shape is perturbed accordingly 
to the field $\theta^h$.
Ongoing investigations focus on the application of the aforementioned CDA 
to the minimization of the compliance discussed in this article.
On the other hand, the choice of explicitly representing the geometry 
and deforming it by moving the computational mesh is responsible for 
the degradation of the final shapes computed by the algorithm.
Currently, we are investigating the approach proposed 
by Allaire \etal \ in \cite{Allaire201422}
that exploits an implicit description of the geometry via a 
level-set function and propagates it by solving an 
Hamilton-Jacobi equation.

Eventually mixed formulations of the linear elasticity problem with 
strongly-enforced symmetry of the stress tensor may be investigated, e.g. the 
Hellinger-Reissner formulation approximated by means of 
Arnold-Winther Finite Element spaces (cf. \cite{MR1930384, 10.2307/40234556})
and the Tangential-Displacement Normal-Normal-Stress (TD-NNS) formulation 
recently proposed by Pechstein and Sch{\"o}berl in \cite{MR2826472}.

\subsection*{Acknowledgements}

Part of this work was developed during a stay of the first author at 
the Laboratoire J.A. Dieudonn\'e at Universit\'e de Nice-Sophia Antipolis 
whose support is warmly acknowledged.

\bibliographystyle{abbrv}
\bibliography{./bibliography}

\end{document}